\documentclass{pspum-l}

\usepackage{amsthm,amsfonts,amsbsy,amssymb,amsmath,amscd}
\usepackage{amsmath}

\usepackage[all]{xy}

\usepackage{hyperref} 
\usepackage{soul}
\usepackage{color} 
\usepackage{typearea} 

\usepackage{pdfcomment} 

\usepackage[mathscr]{eucal}
\usepackage{calrsfs}

\usepackage{verbatim}

\usepackage{tikz-cd}

\newtheorem{teo}{Theorem}[section]
\newtheorem{thm}[teo]{Theorem}
\newtheorem{prop}[teo]{Proposition}
\newtheorem{lemma}[teo]{Lemma}
\newtheorem{cor}[teo]{Corollary}

\newtheorem{conj}[teo]{Conjecture}

\newtheorem{ques}[teo]{Question}
\newtheorem{rmk}[teo]{Remark}

\theoremstyle{Definition}
\newtheorem{defn}[teo]{Definition}

\theoremstyle{remark}

\numberwithin{equation}{section}

\newcommand{\IQbar}{\overline{\mathbb{Q}}}

\newcommand{\BC}{{\mathbb {C}}}

\newcommand{\BE}{{\mathbb {E}}}

\newcommand{\BG}{{\mathbb {G}}}
\newcommand{\BH}{{\mathbb {H}}}

\newcommand{\BQ}{{\mathbb {Q}}}
\newcommand{\BR}{{\mathbb {R}}}
\newcommand{\BS}{{\mathbb {S}}}

\newcommand{\BV}{{\mathbb {V}}}

\newcommand{\BZ}{{\mathbb {Z}}}

\newcommand{\CF}{{\mathcal {F}}}

\newcommand{\CL}{{\mathcal {L}}}

\newcommand{\sD}{{\mathscr{D}}}

\newcommand{\sF}{{\mathscr {F}}}

\newcommand{\sJ}{{\mathscr {J}}}

\newcommand{\sM}{{\mathscr {M}}}

\newcommand{\sO}{{\mathscr {O}}}

\newcommand{\sV}{{\mathscr {V}}}

\newcommand{\sX}{{\mathscr {X}}}

\newcommand{\Aut}{{\mathrm{Aut}}}

\newcommand{\bs}{\backslash}

\newcommand{\Ch}{{\mathrm{Ch}}}

\newcommand{\GL}{{\mathrm{GL}}}

\newcommand{\Hom}{{\mathrm{Hom}}}

\newcommand{\MT}{{\mathrm{MT}}}

\newcommand{\wt}{\widetilde}

\newcommand{\prim}{\mathrm{prim}}

\newcommand{\lra}{{\longrightarrow}}
\newcommand{\iso}{{\overset\sim\lra}}

\renewcommand{\cong}{\simeq}

\begin{document}

\title{Rank of normal functions and Betti strata}

\author{Ziyang Gao}
\address{Department of Mathematics, UCLA, Los Angeles, CA 90095, USA}
\email{ziyang.gao@math.ucla.edu}

\author{Shou-Wu Zhang}
\address{Department of Mathematics, Princeton University, Princeton, NJ 08544, USA}
\email{shouwu@princeton.edu}




\begin{abstract}
In the authors' recent work ``Heights of Ceresa and Gross--Schoen cycles'', we proved the generic positivity of the Beilinson--Bloch heights of the Gross--Schoen and Ceresa cycles. The geometric part of the proof was to prove the maximality of the rank of the associated normal function and the Zariski closedness of the Betti strata.

In this paper, we generalize these geometric results to an arbitrary family of homologically trivial cycles. More generally, we prove a formula to compute the Betti rank and prove the Zariski closedness of the Betti strata, for any admissible normal function of a variation of Hodge structures of weight $-1$. We also define and prove results about degeneracy loci. In the end, we go back to the arithmetic setting and ask some questions about the rationality of the Betti strata and the torsion loci.
\end{abstract}

\maketitle

\section{Introduction}

\subsection{Background: Beilinson--Bloch height and positivity}
Height theory, introduced by Weil in 1928 and further developed by many others, is a fundamental tool in Diophantine and Arithmetic Geometry. For example, the N\'{e}ron--Tate height for points on abelian varieties and its positivity were a key ingredient in the proof of the Mordell--Weil theorem, in Vojta's proof of the Mordell conjecture, and the formulation of the Birch and Swinnerton-Dyer conjecture. It is desirable to extend the definition of heights from points to higher cycles of a projective variety $X$ defined over $\IQbar$.

 In the 1980s, Beilinson \cite{Be} and Bloch \cite{Bl} independently proposed a conditional definition of heights for homologically trivial cycles on $X$ using integral models. 
 They conjectured the positivity of their heights, an extension of the Mordell--Weil theorem, and an extension of the BSD conjecture.  Little is known in this context.

 \medskip
 Recently in \cite{GZ}, we proved the generic positivity of the heights of the Gross--Schoen cycles and the Ceresa cycles, including a height inequality and Northcott property. In these cases  the Beilinson--Bloch heights are known to be well-defined, and their heights have important connections to diophantine geometry and special values of $L$-series.  We also formulated some general conjectures to extend these results to general families of homologically trivial cycles. 

The proof of \cite{GZ} has two parts: the arithmetic part and the geometric part, bridged by Hain's works on the 
 Archimedean local height pairing. The arithmetic part requires the construction of an adelic line bundle over the base, {\it with semi-positive curvature form}, such that the height function is the Beilinson--Bloch height. Another crucial ingredient of the arithmetic part is a volume identity, relating the arithmetic volume and the geometric volume of the generic fiber of the adelic line bundles. Both steps of the arithmetic part are confined to the Gross--Schoen and Ceresa cycles.

{\it The goal of this note is to generalize the geometric part of \cite{GZ} to an arbitrary family}. This geometric part studies the admissible normal function associated with the family of the cycles, and the main results are the Zariski closedness of the Betti strata and the formula to compute the Betti rank. We proved these results for the Gross--Schoen and Ceresa cycles in \cite{GZ}, and in this note we generalize the arguments to an arbitrary family.

\subsection{Setup and main results}\label{SubsectionResultsGeom}
Let $S$ be a quasi-projective variety over $\BC$ and let $(\BV_{\BZ},\sF^{\bullet}) \rightarrow S$ be a variation of Hodge structures (VHS) of  weight $-1$. In practice, every VHS of odd weight can be reduced to the case of weight $-1$ by a suitable Tate twist.

The {\it intermediate Jacobian} is defined to be the family of compact complex tori (write $\sV := \BV_{\BZ}\otimes_{\BZ_S}\sO_S$ for the holomorphic vector associated with $\BV_{\BZ}$)
\begin{equation}\label{EqIntermediateJac}
\pi \colon \sJ(\BV_{\BZ})=\BV_\BZ\bs \sV / \sF^0\sV \lra S.
\end{equation}
Let $\nu$ be an admissible normal function, \textit{i.e.} a holomorphic section of $\pi$ which gives rise to a variation of mixed Hodge structures (VMHS) $\BE_{\nu}$ on $S$ with some nice properties; for a more precise definition see the beginning of $\S$\ref{SubsectionUnivPeriodMap}. 

\smallskip

The fiberwise  isomorphism $\BV_{\BR,s} \iso  \BV_{\BC,s}/\sF^0_s = \sV_s/\sF^0_s$ makes  
 $\sJ(\BV_{\BZ})$ into a local system of real tori
\begin{equation}\label{EqLocSysTori}
\sJ(\BV_{\BZ}) \xrightarrow{\sim} \BV_\BR/\BV_\BZ.
\end{equation}
Let $\CF_{\mathrm{Betti}}$ denote the induced foliation, which we call the {\em Betti foliation}. More precisely, for any point $x\in \sJ(\BV_{\BZ})$, there is a local section  $\sigma \colon U \lra \sJ(\BV_{\BZ})$ 
from a neighborhood $U$ of $\pi(x)$ in $S^{\mathrm{an}}$, with $x \in \sigma(U)$,  represented by a section of the local system $\BV_\BR$. 
 The manifolds $\sigma (U)$ glue together to a foliation $\CF_{\mathrm{Betti}}$ on $\sJ(\BV_{\BZ})$.  In particular, all torsion multi-sections are leaves of $\CF_{\mathrm{Betti}}$.

\medskip

\subsubsection{Betti strata}
For each integer $t \ge 0$, set
\begin{equation}\label{EqBettiStrataSet}
S^{\mathrm{Betti}}(t):=\left\{s\in S(\BC): \dim _{\nu (s)} (\nu(S)\cap \CL_{\nu(s)})\ge t\right\}
\end{equation}
where $\CL_{\nu(s)}$ means the leaf of $\CF_{\mathrm{Betti}}$ passing through $\nu(s)$. 
This subset is by definition real-analytic, but is a priori not clear to be complex analytic since the map \eqref{EqLocSysTori} defining $\CF_{\mathrm{Betti}}$ is a real isomorphism. We then have the following {\em Betti strata} on $S$
\begin{equation}\label{EqBettiStrata}
\emptyset = S^{\mathrm{Betti}}(\dim S +1) \subseteq S^{\mathrm{Betti}}(\dim S) \subseteq \cdots \subseteq S^{\mathrm{Betti}}(1) \subseteq S^{\mathrm{Betti}}(0) = S.
\end{equation}

A main result of this note is the algebraicity of the Betti strata.
\begin{thm}[Theorem~\ref{ThmZariskiClosedDegLoci}]\label{ThmZarClosedIntro}
For each  $t\ge 0$,  $S^{\mathrm{Betti}}(t)$ is Zariski closed in  $S$.
\end{thm} 
When $S$ is a subvariety of $\sM_g$ -- the moduli space of smooth projective curves of genus $g$ -- and $\nu$ is the Gross--Schoen or Ceresa normal function, this is \cite[Thm.~7.2]{GZ}.

\medskip
\subsubsection{Betti rank}
Now assume that $S$ is smooth. Then the Betti foliation induces a decomposition $T_x \sJ(\BV_{\BZ}) = T_x \CF_{\mathrm{Betti}} \oplus T_x (\sJ(\BV_{\BZ})_{\pi(x)}) $ for each $x \in \sJ(\BV_{\BZ})$. Thus for each $s \in S(\BC)$ we have a linear map
\begin{equation}\label{DefnBettiIntro}
\nu_{\mathrm{Betti},s} \colon T_s S \xrightarrow{\mathrm{d}\nu} T_{\nu(s)} \sJ(\BV_{\BZ}) \rightarrow T_{\nu(s)} \sJ(\BV_{\BZ})_s.
\end{equation}
The {\em Betti rank} of $\nu$ is defined to be:
\begin{equation}\label{EqBettiRankIntro}
r(\nu) := \max_{s \in S(\BC)} \dim\nu_{\mathrm{Betti},s}(T_s S).
\end{equation}
A trivial upper bound is $r(\nu) \le \min\{\dim S, \frac{1}{2} \dim \BV_{\BQ,s} \}$ for any $s \in S(\BC)$. 

The Betti rank is easily seen to be related to the Betti strata, for example 
\begin{equation}\label{EqS1BettiRank}
S^{\mathrm{Betti}}(1)\not= S \Longleftrightarrow r(\nu) = \dim S. 
\end{equation}

As explained in $\S$\ref{SubsectionUnivPeriodMap}, the VMHS $\BE_{\nu}$ on $S$ induces a period map $\varphi =\varphi_{\nu} \colon S\lra \Gamma \bs \sD$, with $\sD$ a mixed Mumford--Tate domain whose Mumford--Tate group we denote by $\mathbf{G}$ (so $\sD$ is a $\mathbf{G}(\BR)^+$-orbit and $\mathbf{G}$ is the generic Mumford--Tate group of the VMHS $\BE_{\nu}$).

Denote by $V$ the unipotent radical of $\mathbf{G}$; see Remark~\ref{RmkSubTorusFib} for its geometric meaning. Now the trivial upper bound on $r(\nu)$ can be easily improved to be
\begin{equation}\label{EqBettiRankTrivialUpperBound}
r(\nu) \le \min\{\dim\varphi(S), \frac{1}{2}\dim_{\BQ}V\};
\end{equation}
see \eqref{EqBettiRankTrivialUB}.

Another main result is the following formula for $r(\nu)$, which is often computable in practice. 
\begin{thm}[Theorem~\ref{ThmBettiRankFormulaMain}]\label{ThmBettiRankIntro} 
The Betti rank is given by
\begin{equation}\label{EqBettiRankFormulaIntro}
r(\nu)=\min _N \left (\dim \varphi_{/N}(S)+\frac 12 \dim_\BQ (V\cap N)\right),
\end{equation}
where $N$ runs over all normal subgroups of $\mathbf{G}$, and $\varphi_{/N}$ is the induced period map
\[
\varphi_{/N}: S\xrightarrow{\varphi} \Gamma\bs\sD \xrightarrow{[p_N]} \Gamma _{/N}\bs (\sD/N)
\]
with $[p_N]$ the quotient by $N$ defined in \eqref{EqQuotientMTDomain}.
\end{thm}

We will give two applications of Theorem~\ref{ThmBettiRankIntro}  where the trivial upper bound \eqref{EqBettiRankTrivialUpperBound} on $r(\nu)$ is attained: 
Theorem~\ref{ThmBignessGeom} when $(\BV_{\BZ},\sF^{\bullet}) \rightarrow S$ is irreducible, and Theorem~\ref{ThmBignessGeom2} when the {\em algebraic monodromy group of $(\BV_{\BZ},\sF^{\bullet}) \rightarrow S$ is simple} (or when the Mumford--Tate group of $(\BV_{\BZ},\sF^{\bullet}) \rightarrow S$ is simple). 

When applied to the Gross--Schoen or Ceresa normal function over the whole $\sM_g$ ($g\ge 3$), we get $r(\nu) = 3g-3 = \dim \sM_g$ as in \cite[$\mathsection$7]{GZ}. In this particular case, Hain \cite{Ha24} gave a different proof.

\begin{rmk}\label{RmkTorsion}
In the course of proof, we also obtain geometric results on torsion loci. See Corollary~\ref{CorSF1Torsion} and Remark~\ref{RmkSF1Torsion}. Some of these results are also independently proved by Kerr--Tayou \cite{KerrTayou} with a different method.
\end{rmk}
When applied to arithmetic families, the results mentioned in Remark~\ref{RmkTorsion} apply to {\it non-$\IQbar$ points}. To study torsion $\IQbar$-points, one needs to study the {\it degeneracy loci} as indicated by \cite{GH}. We prove results about the degeneracy loci in $\mathsection$\ref{SectionDegLoci}.


\subsection{Organization of the paper}
We will give a quick summary on classifying spaces of mixed Hodge structures and Mumford--Tate domains in $\S$\ref{SectionSummaryClassifyingSpaceMTDomain}. Then in $\S$\ref{SectionPeriodMap} we recall the definition of the period map associated with a normal function  and explain how to see the Betti foliation in this language.  With these preparations both Theorem~\ref{ThmZarClosedIntro} and Theorem~\ref{ThmBettiRankIntro} are proved in $\S$\ref{SectionBettiRank}. Then we give two applications of this formula in geometric situations in $\S$\ref{SectionGeomApp}, with a brief discussion on their applications to the  torsion loci.

In $\S$\ref{SectionDegLoci}, we discuss the degeneracy loci. In $\S$\ref{SectionFuture}, we ask some conjectures and questions about the Betti strata, which are closely related to our program of studying the positivity of Beilinson--Bloch heights.

\subsection*{Acknowledgements}
The material presented in this paper is part of the authors' program to study the positivity of the Beilinson--Bloch height, and  was done when ZG visited SZ at Princeton University in January, February, and May 2024. ZG would like to thank Princeton University for its hospitality. We would like to thank the referee for their comments which help us improve the writing.


Ziyang Gao received funding from the European Research Council (ERC) under the European Union’s Horizon 2020 research and innovation program (grant agreement n$^{\circ}$ 945714) during this work. Shouwu Zhang is supported by an NSF award, DMS-2101787.

\section{Quick summary on classifying spaces and Mumford--Tate domains}\label{SectionSummaryClassifyingSpaceMTDomain}
In this section, we give a  review of classifying spaces of mixed Hodge structures and Mumford--Tate domains.

Let $V_0$ be a finite-dimensional $\BQ$-vector space. Let $E$ be a $\BQ$-vector space fitting into a short exact sequence $0 \rightarrow V_0 \rightarrow E \rightarrow \BQ \rightarrow 0$. This short exact sequence must split, and we fix a splitting 
 $E = V_0 \oplus \BQ$.

\subsection{Classifying spaces}\label{SubsectionClassifyingSpace}
\subsubsection{Pure Hodge structures}  Consider the polarized Hodge data on $V_0$: a non-degenerate skew pairing $Q_{-1} \colon V_0 \otimes V_0 \rightarrow \BQ(1)$, and a partition $\{h^{p,q}_{V_0}\}_{p,q\in \BZ}$ of $\dim V_{0,\BC}$ into non-negative integers with $p+q=-1$ such that $h^{p,q}_{V_0} = h^{q,p}_{V_0}$. Then there exists a {\em classifying space} $\sM_0$ parametrizing $\BQ$-Hodge structures on $V_0$ of weight $-1$ with a polarization by $Q_{-1}$ such that the $(p,q)$-constituent of $V_{0,\BC}$ has complex dimension $h^{p,q}$. Moreover, consider the $\BQ$-group  $\mathbf{G}_0^{\sM} := \mathbb{G}_{\mathrm{m}} \cdot \mathrm{Aut}( V_0 ,Q_{-1}) < \mathrm{GL}(V_0)$, where $\mathbb{G}_{\mathrm{m}}(\BQ) = \BQ^{\times}$ acts on $V_0$ by scalar multiplication. The associated real Lie group $\mathbf{G}_0^{\sM}(\BR)^+$ acts transitively on $\sM_0$, \textit{i.e.} $
\sM_0 = \mathbf{G}_0^{\sM}(\BR)^+ x_0$ 
for any point $x_0 \in \sM_0$. This makes $\sM_0$ into a semi-algebraic open subset of a flag variety $\sM_0^\vee$, which is a suitable $\mathbf{G}_0^{\sM}(\BC)$-orbit, and hence endows $\sM_0$ with a complex structure and a semi-algebraic structure.

We can be more explicit on the action of $\mathbf{G}_0^{\sM}(\BR)^+$ on $\sM_0$. For each $x_0 \in \sM_0$, we have a Hodge decomposition and a Hodge filtration $F^{\bullet}_{x_0}$
\begin{equation}\label{EqHodgeDecomPure}
V_{\BC} = \bigoplus\nolimits_{p+q=-1} (V_{x_0})^{p,q}, \qquad F^p_{x_0} V_{\BC} = \bigoplus\nolimits_{p'\ge p} (V_{x_0})^{p',q'}
\end{equation}
with $(V_{x_0})^{q,p} = \overline{(V_{x_0})^{p,q}}$. 
The inclusion $\sM_0 \subseteq \sM_0^\vee$ is given by $x_0 \mapsto F^{\bullet}_{x_0}$. 

For the Deligne torus $\BS = \mathrm{Res}_{\BC/\BR}\BG_{\mathrm{m},\BC}$, 
the bi-grading decomposition above  defines a morphism $h_{x_0} \colon \BS \rightarrow \mathrm{GL}(V_{\BR})$, with $(V_{x_0})^{p,q}$ the eigenspace of the character $z \mapsto z^{-p}\bar{z}^{-q}$ of $\BS$. It is known that $h_{x_0}(\BS) < \mathbf{G}_0^{\sM}(\BR)$ for all $x_0 \in \sM_0$. Hence we have  a $\mathbf{G}_0^{\sM}(\BR)^+$-equivariant map, which is known to be injective,
\begin{equation}
\sM_0  \rightarrow \Hom(\BS, \mathbf{G}^{\sM}_{0,\BR}), \quad  x_0 \mapsto  h_{x_0} 
\end{equation}
with the action of $\mathbf{G}_0^{\sM}(\BR)^+$ on $\Hom(\BS, \mathbf{G}^{\sM}_{0,\BR})$ given by conjugation. So we will view $\sM_0$ as a subset of $\Hom(\BS, \mathbf{G}^{\sM}_{0,\BR})$.

For each $x_0 \in \sM_0$, denote by $\mathrm{MT}(x_0)$ the Mumford--Tate subgroup of the Hodge structure parametrized by $x_0$. It is a reductive subgroup of $\mathbf{G}_0^{\sM}$.

\subsubsection{Mixed Hodge structures} 
Next, we turn to mixed Hodge structures of weight of $-1$ and $0$. 
Fix the following data on $E = V_0 \oplus \BQ$: the weight filtration 
$$W_\bullet := (0 = W_{-2}E \subseteq W_{-1}E = V_0 \subseteq W_0 E = E);$$
the partition $\{h^{p,q}\}_{p,q\in \BZ}$ of $\dim E_{\BC}$ into non-negative integers, with $h^{p,q} = h^{p,q}_{V_0}$ for $p+q=-1$ and $h^{0,0} = 1$ and $h^{p,q} = 0$ otherwise; a non-degenerate skew pairing $Q_{-1} \colon V_0 \otimes V_0 \rightarrow \BQ(1)$.

There exists the \textit{classifying space} $\sM$ parametrizing $\BQ$-mixed Hodge structures $(E, W_\bullet, F^\bullet)$ of weight $-1$ and $0$ such that: 
\begin{enumerate}
\item[(a)] the $(p,q)$-constituent $ \mathrm{Gr}^p_F \mathrm{Gr}^W_{p+q} E_\BC$ has complex dimension $h^{p,q}$;
\item[(b)] 
$\mathrm{Gr}^W_{-1} E = V_0$ is polarized by $Q_{-1}$ (so $E$ is {\em graded-polarized} because $\mathrm{Gr}^W_0 E = \BQ$ is polarized by $\BQ \otimes \BQ \rightarrow \BQ$, $a\otimes b\mapsto ab$).
\end{enumerate}
See \cite[below (3.7) to the Remark below Lem.~3.9]{Pearlstein2000}.

In our case, we need a better understanding of the structure of $\sM$. 
Consider the $\BQ$-group $\mathbf{G}^{\sM} :=\Aut (E, Q_{-1}, W)$. Then we have
\begin{equation}\label{EqMTGroupClassSpMixed}
\mathbf{G}^{\sM} = V_0 \rtimes \left( \mathbb{G}_{\mathrm{m}} \cdot \mathrm{Aut}(V_0,Q_{-1}) \right) = V_0 \rtimes \mathbf{G}_0^{\sM}
\end{equation}
because each  $v\in V_0$ can be seen as an element of $\GL(E)$ by sending $(w, a)\in V_0\oplus \BQ$ to $(w+av, a)$. 
The map $x \mapsto F_x^{\bullet}$ realizes $\sM$ as a semi-algebraic open subset of a suitable flag variety $\sM^\vee$, which is an orbit under $\mathbf{G}^{\sM}(\BC)$.  Moreover, the mixed Hodge structure parametrized by each $x \in \sM$ is split over $\BR$;\footnote{Here it is crucial that the mixed Hodge structure has two adjacent weights!} more precisely, there exists a splitting $E_{\BR} = V_{0,\BR} \oplus \BR$ of Hodge structures given by the inverse of $E_{\BR} \cap F_x^0 E_{\BC} \iso V_{0,\BR}$. So (see for example  \cite[last~Remark~of~$\mathsection$3]{Pearlstein2000})
\begin{equation}\label{EqClassifyingSpaceMixed}
\sM = \mathbf{G}^{\sM}(\BR)^+ x.
\end{equation}
This makes $\sM$ into a semi-algebraic open subset of a flag variety $\sM^\vee$ and hence endows $\sM$ with a complex structure and a semi-algebraic structure.

For each $x \in \sM$, denote by $\mathrm{MT}(x)$ the Mumford--Tate subgroup of the Hodge structure parametrized by $x$. It is a subgroup of $\mathbf{G}^{\sM}$ which is not necessarily reductive.

\subsubsection{Fibered structure}
For each $x \in \sM$, we have an induced Hodge structure on $V_0$ which we denote by $p(x)$. This induces a projection
\begin{equation}\label{EqQuotClassfyingSpaceToPure}
p \colon \sM \rightarrow \sM_0,
\end{equation}
which is compatible with the quotient $\mathbf{G}^{\sM} \rightarrow \mathbf{G}_0^{\sM} = \mathbf{G}^{\sM}/V_0$. 

Moreover, \eqref{EqMTGroupClassSpMixed} induces a semi-algebraic isomorphism
\begin{equation}\label{EqClassfyingSpaceDecomposition}
(q,p) \colon \sM \iso V_0(\BR)\times \sM_0.
\end{equation}
The complex structure on $\sM$ is given by making $\sM$ into a vector bundle over $\sM_0$ with fibers $V_{0,\BR} \iso V_{0,\BC}/F^0_xV_{0,\BC}$. For each $a \in V_0(\BR)$, the submanifold $(q, p)^{-1}(\{a\}\times \sM_0)$ is a semi-algebraic and  complex submanifold of $\sM$. More precisely, let us summarize into:
\begin{prop}\label{PropStruMT}
Let $F^0(V_0 \times \sM_0)$ be the holomorphic vector bundle on $\sM_0$ whose fiber over $x_0 \in \sM_0$ is $F^0_{x_0}V_{0,\BC}$. We have:
\begin{enumerate}
\item[(i)] The semi-algebraic structure on $\sM \subseteq \Hom(\BS, \mathbf{G}^{\sM}_{\BR})$ is given by 
\begin{equation}\label{EqBettiMTDomainSemiAlg}
V_0(\BR) \times \sM_0 \xrightarrow{\sim} \sM, \quad (v, x_0) \mapsto \mathrm{Int}(v)\circ h_{x_0}.
\end{equation}
\item[(ii)] The complex structure on $\sM$ is given by  $\sM \iso (V_0(\BC) \times \sM_0)/F^0(V_0 \times \sM_0)$.
\item[(iii)]
These two structures are related by the natural bijection
\begin{equation}\label{EqBettiMTDomain}
V_0(\BR) \times \sM_0\subseteq V_0(\BC) \times \sM_0 \longrightarrow   (V_0(\BC) \times \sM_0)/F^0(V_0 \times \sM_0).
\end{equation}
\end{enumerate}
\end{prop}
\begin{proof}
Fix $x \in \sM$ and let $h_x \colon \BS \rightarrow \mathbf{G}^{\sM}_{\BR}$ be the corresponding homomorphism.

By \cite[1.8(a)]{PinkThesis}, for any $x \in \sM$ we have $\mathrm{Stab}_{\mathbf{G}(\BR)^+}(x) \cong \mathrm{Stab}_{\mathbf{G}_0(\BR)^+}(\pi(x))$ for $\pi \colon \sM \rightarrow \sM_0$. Hence part (i) is a direct consequence of \eqref{EqMTGroupClassSpMixed} that $\mathbf{G}^{\sM} = V_0\rtimes \mathbf{G}_0^{\sM}$. 

 Set $\sX := \mathbf{G}^{\sM}(\BR)^+V_0(\BC)\cdot h_x \subseteq \Hom(\BS_{\BC},\mathbf{G}^{\sM}_{\BC})$, where the action is given by conjugation.  
The quotient $p \colon \mathbf{G}^{\sM} \rightarrow  \mathbf{G}_0^{\sM}$ induces a natural surjective map $\sX \rightarrow \sM_0$, and by \cite[1.8(a)]{PinkThesis} each fiber of this map is a $V_0(\BC)$-torsor.

The natural morphism of algebraic groups $\mathbf{G}_0^{\sM} \cong \{0\}\times \mathbf{G}_0^{\sM} < V_0\rtimes \mathbf{G}_0^{\sM} = \mathbf{G}^{\sM}$  induces a global section of $\sX \rightarrow \sM_0$, and hence an isomorphism $\sX \cong V_0(\BC) \times \sM_0$ over $\sM_0$. 

Consider the following surjective equivariant  map 
 \begin{equation}\label{EqHodgeMT}
 \varphi \colon \sX \rightarrow \sD, \qquad g h_x g^{-1} \mapsto g\cdot x.
 \end{equation}
By \cite[1.8(b)]{PinkThesis}, for each $x \in \sD$, the fiber $\varphi^{-1}(x)$ is  a principle homogeneous space under $F^0_{x_0}V_{0, \BC}$ with $x_0=p(x)$. Hence $\varphi$ gives a bijection $(V_0(\BC) \times \sM_0)/F^0(V_0(\BC) \times \sM_0) \rightarrow \sM$ over $\sM_0$. Moreover, the complex structure of $\sM$ is precisely given by this bijection; see \cite[Proof of Prop.~2.6 in Appendix~A]{GKAS} which is a consequence of \cite[Thm.~3.13]{Pearlstein2000}. This establishes (ii).

For (iii), let $\sX_{\BR} := \mathbf{G}^{\sM}(\BR)^+\cdot h_x \subseteq \sX$. Then \eqref{EqBettiMTDomain} is $\sX_{\BR} \subseteq \sX \xrightarrow{\varphi}\sM$. We are done.
\end{proof}


\subsection{Mumford--Tate domains and quotients}\label{SubsectionMTDomain}
\begin{defn}\label{DefnMTDomain}
 A subset $\sD$ of the classifying space $\sM$ is called a {\em (mixed) Mumford--Tate domain} if there exists an element $x \in \sD$ such that $\sD = \mathbf{G}(\BR)^+ x$, where $\mathbf{G} = \mathrm{MT}(x)$.
\end{defn}
Let $\sD$ be a Mumford--Tate domain in $\sM$.\footnote{$\sM$ is a Mumford--Tate domain in itself with $\mathrm{MT}(\sM) = \mathbf{G}^{\sM}$.} Then $\sD$ is both complex analytic irreducible and semi-algebraic in $\sM$.
The group $\mathbf{G}$ in the definition above is called the {\em generic Mumford--Tate group of $\sD$} and is denoted by $\mathrm{MT}(\sD)$. We have $\mathrm{MT}(x) < \mathbf{G}$ for all $x \in \sD$. 

It is known that $\mathbf{G} < \mathbf{G}^{\sM}$, and the unipotent radical of $\mathbf{G}$ equals $V:= V_0 \cap \mathbf{G}$ (by reason of weight). 
Let $\mathbf{G}_0 := \mathbf{G}/V$ be the reductive part. Set $\sD_0 := p(\sD) \subseteq \sM_0$ for the map $p$ defined in \eqref{EqQuotClassfyingSpaceToPure}. Then $\sD_0$ is a $\mathbf{G}_0(\BR)^+$-orbit and is in fact a (pure) Mumford--Tate domain in the classifying space $\sM_0$.

\begin{defn}\label{DefnAlgMTDomain}
A subset of $\sD$ is said to be {\em irreducible algebraic} if it is both complex analytic irreducible and semi-algebraic. 
\end{defn}
In view of \cite[Lem.~B.1 and its proof]{KlinglerThe-Hyperbolic-}, a subset of $\sD$ is irreducible algebraic if and only if it is a component of $U\cap \sD$ with $U$ an algebraic subvariety of $\sM^\vee$.

We also have a similar description of the fibered structure of $p \colon \sD \rightarrow \sD_0$ as in \eqref{EqClassfyingSpaceDecomposition}. Each $x_0 \in \sD_0 \subseteq \sM_0$ endows $V_0$ with a Hodge structure of weight $-1$ whose Mumford--Tate group is a subgroup of $\mathbf{G}_0$, and $V$ is a sub-Hodge structure of $V_0$ because $V$ is a $\mathbf{G}_0$-submodule of $V_0$ and $\mathrm{MT}(x_0) < \mathbf{G}_0$. 
The following lemma is well-known to experts, but we cannot find any reference. We give a proof here  to make our paper more complete. See also \cite[Lem.~6.2]{GZ}.
\begin{lemma}\label{LemmaStrucMTDomain}
Under the isomorphism \eqref{EqClassfyingSpaceDecomposition}, the image of the subset $\sD \subseteq \sM$ is $(V(\BR)+v_0) \times \sD_0$ for some $v_0 \in V_0(\BQ)$.
\end{lemma}
We thus obtain a natural isomorphism $\sD \cong V(\BR) \times \sD_0$. 
The following corollary is then an easy consequence of the last sentence of $\S$\ref{SubsectionClassifyingSpace}.
\begin{cor}\label{CorZariskiClosureBettiFiber}
Let $\wt{Z}_0 \subseteq \sD_0$ be an irreducible algebraic subset. Then for any $a \in V(\BR)$, the subset $\{a\} \times \wt{Z}_0  \subseteq V(\BR) \times \sD_0\cong \sD$ is  irreducible algebraic.
\end{cor}

\begin{proof}[Proof of Lemma~\ref{LemmaStrucMTDomain}]
Recall the projection $p\colon \sM \rightarrow \sM_0$.  
To ease notation in this proof, we will write $\MT_x$ for the Mumford--Tate group of the Hodge structure parametrized by $x$ for any point in $\sM$ or in $\sM_0$.

For each $x\in \sM $, we have a surjective morphism $\MT_x\lra \MT _{p(x)}$ with kernel $V_x:=V_0\cap \MT _x$. Thus $V_x$ is the unipotent radical of $\MT_x$. 
This surjection  has a section $s\colon \MT _{p(x)}\lra \MT _x$. Via $\mathbf{G}^{\sM} = V_0\rtimes \mathbf{G}_0^{\sM}$,
 this section can be written as $s(g)=(\sigma (g), g)$ with $\sigma $ a cocycle in $Z^1(\MT _{p(x)}, V_0)$.
But  $H^1(\MT_{p(x)}, V_0)=0$ since $\MT_{p(x)}$ is reductive. 
 Thus $\sigma$  is 
 a coboundary, \textit{i.e.} $\sigma (g)=v_0-gv_0$ for some $v_0\in V_0(\BQ)$ unique up to addition by invariant vectors of $V_0$ under action by 
 $\MT _{p(x)}$. Thus, we get a precise Levi decomposition
 \begin{equation}\label{eq-MTx}
 \MT _x=V_x\rtimes \MT _{x_0}=\mathrm{Int}(v_0)(V_x\rtimes \MT _{p(x)}),
 \end{equation}
 where $V_x=V\cap \MT _x$, $x_0=(v_0, p(x))$ is a point of $\sM$ via the isomorphism \eqref{EqBettiMTDomainSemiAlg}.
 With this isomorphism, we see that $x=(v_1+v_0, p(x))$ with $v_1\in V(\BR)$ which has Zariski closure $V_x$ over $\BQ$.
 
Now let us go back to our Mumford--Tate domain  $\sD = \mathbf{G}(\BR)^+x$  with $\mathbf{G} = \MT_x$. Then $\MT_{p(x)} = \mathbf{G}_0$ and $V_x = V$. Thus \eqref{eq-MTx} becomes $\mathbf{G} = \mathrm{Int}(v_0) (V\rtimes \mathbf{G}_0)$. Now the conclusion follows from \eqref{EqBettiMTDomainSemiAlg}.
\end{proof}

\subsection{Quotient by a normal subgroup}
Next we discuss   quotients of $\sD$.  
%
Let $N \lhd \mathrm{MT}(\sD)$. By \cite[Prop.~5.1]{GKAS}, we have a quotient in the category of complex varieties $
p_N \colon \sD \rightarrow \sD/N$ 
with the following properties: (i) $\sD/N$ is a Mumford--Tate domain in some classifying space of mixed Hodge structures (which must be of weight $-1$ and $0$) and $\mathrm{MT}(\sD/N) = \mathbf{G}/N$; (ii) each fiber of $p_N$ is an $N(\BR)^+$-orbit. Moreover $p_N$ is semi-algebraic. 

Assume $\Gamma$ is an arithmetic subgroup of $\mathbf{G}(\BQ)^+$. Then, the quotient $\Gamma\backslash \sD$ is an orbifold. Denote by $\Gamma_{/N}$ the image of $\Gamma$ under the quotient $\mathbf{G} \rightarrow \mathbf{G}/N$. Then $p_N$ induces
\begin{equation}\label{EqQuotientMTDomain}
[p_N] \colon \Gamma\backslash \sD \rightarrow \Gamma_{/N} \backslash (\sD/N).
\end{equation}

We close this section with the following definition.
\begin{defn}\label{DefnWSP}
A subset $\sD_N$ of $\sM$ is called a {\em weak Mumford--Tate domain} if there exist $x \in \sM$ and a normal subgroup $N$ of $\mathrm{MT}(x)$ such that $\sD_N = N(\BR)^+x$.
\end{defn}
In this definition, if $x$ is taken to be a Hodge generic point, \textit{i.e.} $\mathrm{MT}(x) = \mathbf{G}$, then the weak Mumford--Tate domain thus obtained is a fiber of $p_N$.

\section{Period map associated with normal functions}\label{SectionPeriodMap}
Retain the notation of $\S$\ref{SubsectionResultsGeom}. 
Let $S$ be a smooth irreducible quasi-projective variety. Let $(\BV_{\BZ},\sF^{\bullet})$ be a polarized VHS on $S$ of weight $-1$. Let $\nu \colon S \rightarrow  \sJ(\BV_{\BZ})$ be an admissible normal function (defined in $\mathsection$\ref{SubsectionUnivPeriodMap}).

The first goal of this section is to recall the {\em period map} $\varphi_{\nu}: S\to \Gamma \bs \sD$ associated with $\nu$ (when $\nu$ is clearly in context, we simply denote it by $\varphi$), where 
 $\sD$ is a (mixed) Mumford--Tate domain with generic Mumford--Tate group $\mathbf{G}$, and 
 $\Gamma$ is a suitable arithmetic subgroup of $\mathbf{G}(\BQ)$. This is described in $\S$\ref{SubsectionUnivPeriodMap}.

The second goal of this section is to explain how to see the Betti foliation on $\sJ(\BV_{\BZ})$ in terms of the period map and the fibered structure of the Mumford--Tate domain $\sD$, or as called in references, how to see the fibration \eqref{EqQuotClassfyingSpaceToPure} of the classifying space $\sM \rightarrow \sM_0$ as the {\em universal intermediate Jacobians}. This is done in $\mathsection$\ref{SubsectionBettiFoliationBettiRank}.

We also recall in $\mathsection$\ref{SubsectionOminimality} the o-minimal structure attached to the period map.
\subsection{Period map}\label{SubsectionUnivPeriodMap}
Admissible normal functions are defined as follows. First of all, each holomorphic section $\nu \colon S \rightarrow \sJ(\BV_{\BZ})$ defines a family of mixed Hodge structures $(\BE_{\nu},W_{\bullet}, \sF_{\BE}^{\bullet})$ on $S$ of weight $-1$ and $0$, such that the local system $\BE_{\nu}$ fits into a short exact sequence $0 \rightarrow \BV_{\BZ} \rightarrow \BE_{\nu} \rightarrow \BZ_S \rightarrow 0$; see for example \cite[$\S$4.1]{Ha13}. Then $\nu$ is called an {\it admissible normal function} if $(\BE_{\nu},W_{\bullet}, \sF_{\BE}^{\bullet})$ is an admissible variation of mixed Hodge structures (in particular, graded-polarizable).\footnote{Admissibility is a good asymptotic property near the boundary \cite{SZ85, Kashiwara86}. The precise definition is not needed in our paper since all our VMHSs are of geometric origin (hence admissible \cite{EZ}).} 

The VMHS $(\BE_{\nu},W_{\bullet}, \sF_{\BE}^{\bullet})$ induces the {\it universal period map}
\[
\wt{\varphi} = \wt{\varphi}_{\nu} \colon \wt{S} \longrightarrow \sM 
\]
from the universal cover $u_S \colon \wt{S} \rightarrow S^{\mathrm{an}}$ 
to a suitable classifying space $\sM$ \eqref{EqClassifyingSpaceMixed} of mixed Hodge structures of weight $-1$ and $0$ as follows. There is canonical trivialization $u_S^*\BE_{\nu} \cong \wt{S} \times E_{\nu}$ with $E_{\nu} = H^0(\wt{S}, u_S^*\BE_{\nu})$; then for any $s \in S(\BC)$ the fiber $\BE_{\nu,s}$ can be canonically identified with $E_{\nu}$. Then $\wt{\varphi}$ sends  $\tilde{s}\mapsto (E_{\nu,\BQ}, (W_{\bullet})_{\wt{s}}, (\sF_{\BE}^{\bullet})_{\wt{s}})$.

Next, let $\sD$ be the {\it smallest Mumford--Tate domain} in $\sM$ which contains $\wt{\varphi}(\wt{S})$; see for example \cite[$\mathsection$7.1]{GKAS} for a construction of $\sD$. Then $\wt{\varphi}$ becomes $\wt{\varphi} \colon \wt{S} \rightarrow \sD$. Finally, since $S$ is a quasi-projective variety, there exists an arithmetic subgroup $\Gamma$ of $\mathbf{G}(\BQ)$ such that $\wt{\varphi}$ descends to a holomorphic map $\varphi \colon S^{\mathrm{an}} \rightarrow \Gamma\backslash\sD$.


 \subsection{Setup for o-minimality}\label{SubsectionOminimality}

Let $\mathfrak{F} \subseteq \sD$ be a fundamental set for the quotient $u \colon \sD \rightarrow \Gamma\backslash\sD$, \textit{i.e.} $u|_{\mathfrak{F}}$ is surjective and $(u|_{\mathfrak{F}})^{-1}(\bar{x})$ is finite for each $\bar{x} \in \Gamma\backslash\sD$. If $\mathfrak{F}$ is a semi-algebraic subset of $\sD$, then $u|_{\mathfrak{F}}$ induces a semi-algebraic structure on $\Gamma\backslash\sD$.

By the main result of  \cite{BBKT}, there exists a semi-algebraic fundamental set $\mathfrak{F} \subseteq \sD$ for the quotient $u \colon \sD \rightarrow \Gamma\backslash\sD$ with the following properties: $\varphi$ is $\BR_{\mathrm{an},\exp}$-definable for the semi-algebraic structure on $\Gamma\backslash\sD$ defined by $\mathfrak{F}$.

\subsection{Relating the Betti foliation and the period map}\label{SubsectionBettiFoliationBettiRank}
Let $V_0 := H^0(\wt{S}, u_S^*\BV_{\BQ})$. Now $V_{0,\BZ} := H^0(\wt{S},u_S^*\BV_{\BZ})$ is a lattice in $V_0(\BR)$ and  $u_S^*\BV_{\BZ} \cong V_{0,\BZ} \times \wt{S}$. So $\sJ(\BV_{\BZ}) = \BV_{\BR}/\BV_{\BZ}$ induces $\sJ(\BV_{\BZ}) \times_S \wt{S} =  (V_{0,\BR}/V_{0,\BZ}) \times \wt{S} $. For the universal covering map $u_{\sJ} \colon \tilde{\sJ} \rightarrow \sJ(\BV_{\BZ})$, this furthermore induces $\tilde{\sJ} =  V_0(\BR) \times \wt{S}$. Then  the leaves of the Betti foliation $\CF_{\mathrm{Betti}}$ defined below \eqref{EqLocSysTori} are precisely those $u_{\sJ}(\{a\} \times \wt{S})$ for all $a \in V_0(\BR)$.

The leaves of $\CF_{\mathrm{Betti}}$ can be furthermore written in terms of the period maps as follows. For each $s \in S(\BC)$ the fiber $\BV_{\BQ,s}$ can be canonically identified with $V_0$, and the polarized VHS $\BV_{\BZ} \rightarrow S$ induces a period map $\wt{\varphi}_0 \colon \wt{S} \rightarrow \sM_0$. 
On the other hand, each $x \in \sJ(\BV_{\BZ})$ lies in $\sJ(\BV_{\BZ})_s = \sJ(\BV_{\BZ,s})$ for $s =\pi(x)$, which is canonically isomorphic to $\mathrm{Ext}_{\mathrm{MHS}}(\BZ(0), \BV_{\BZ,s})$ by Carlson \cite{Carlson85}.  Hence each $\wt{x} \in \wt{\sJ}$ gives rise to a $\BQ$-mixed Hodge structure of weight $-1$ and $0$ which is an extension of $\BQ(0)$ by $V_0$. 
So we obtain a period map $\wt{\varphi}_{\sJ} \colon \wt{\sJ} \rightarrow \sM$, which by construction is
\begin{equation}\label{EqPeriodJacobian}
\wt{\varphi}_{\sJ} \colon \wt{\sJ} = V_0(\BR) \times \wt{S} \xrightarrow{(1,  \wt{\varphi}_0)} V_0(\BR) \times \sM_0 = \sM
\end{equation}
under the identifications $\tilde{\sJ} =  V_0(\BR) \times \wt{S}$ given in the previous paragraph and $\sM = V_0(\BR) \times \sM_0$ given by \eqref{EqClassfyingSpaceDecomposition}. 
By last paragraph, the leaves of $\CF_{\mathrm{Betti}}$ are precisely $u_{\sJ}\left( \wt{\varphi}_{\sJ}^{-1}(\{a\}\times \sM_0) \right)$ for $a \in V_0(\BR)$.

\smallskip
Now we go back to $\wt{\varphi} = \wt{\varphi}_{\nu} \colon \wt{S} \rightarrow \sD$ given by the normal function $\nu$. 
The map $\nu \circ u_S \colon \wt{S} \rightarrow S \rightarrow \sJ(\BV_{\BZ})$ lifts to $\wt{\nu} \colon \wt{S} \rightarrow \wt{\sJ}$, and $\wt{\varphi} = \wt{\varphi}_{\sJ}\circ \wt{\nu}$. 

Denote by $\mathbf{G}:=\mathrm{MT}(\sD)$ the generic Mumford--Tate group (see below Definition~\ref{DefnMTDomain}), and $V$ its unipotent radical. 

Denote by $\sD_0 := p(\sD) \subseteq \sM_0$ for the projection to the pure part $p \colon \sM \rightarrow \sM_0$  from \eqref{EqQuotClassfyingSpaceToPure}, and by abuse of notation $p =p|_{\sD} \colon \sD \rightarrow \sD_0$. Recall the natural isomorphism $\sD = V(\BR) \times \sD_0$ below Lemma~\ref{LemmaStrucMTDomain}. The period maps $\varphi$ and $\wt\varphi$  can be completed into:
\begin{equation}\label{EqDiagramPeriodMapComplete}
\xymatrix{
\wt{S} \ar[r]_-{\wt{\varphi}} \ar[d]_{u_S} \ar@/^1.5pc/[rr]|-{\wt{\varphi}_0} & \sD = V(\BR) \times \sD_0  \ar[d]^-{u}  \ar[r]_-{p} & \sD_0 \ar[d]^-{u_0} \\
S \ar[r]^-{\varphi}  \ar@/_1.5pc/[rr]|-{\varphi_0} & \Gamma\backslash \sD \ar[r]^-{[p]} & \Gamma_0 \backslash \sD_0 
}
\end{equation}
with $p$ being the projection to the second factor.

We can decompose $\wt{\varphi}$ as 
\begin{equation}
\wt{\varphi} = (\wt{\varphi}_V , \wt{\varphi}_0) \colon \wt{S} \longrightarrow  \sD = V(\BR) \times \sD_0.
\end{equation}

Recall $S^{\mathrm{Betti}}(t) := \{s \in S(\BC): \dim_{\nu(s)} (\nu(S) \cap \CF_{\mathrm{Betti}}) \ge t \}$ and define $$S_{\mathrm{rk}}(t) = \{ s\in S(\BC): \dim \nu_{\mathrm{Betti},s}(T_s S) \le t\}$$ for $\nu_{\mathrm{Betti},s} \colon T_s S \rightarrow T_{\nu(s)} \sJ(\BV_{\BZ,s})$ defined in \eqref{DefnBettiIntro}.

\begin{lemma}\label{LemmaBettiFoliationBettiRank}
For each $t \ge 0$, we have
\begin{align}\label{EqBettiMapBettiFoliation1}
u_S^{-1}\left(S^{\mathrm{Betti}}(t)\right) & = \{\wt{s} \in \wt{S} : \dim_{\wt{s}} \wt{\varphi}^{-1}(\{\wt{s}_V\} \times \sD_0) \ge t \} \nonumber\\
& =  \bigcup\nolimits_{r\ge 0} \{\wt{s} \in \wt{S} : \dim_{\wt{s}} \wt{\varphi}^{-1}(\wt{\varphi}(\wt{s})) = r, ~ \{\wt{s}_V\} \times \wt{C} \subseteq \wt{\varphi}(\wt{S}) \text{ for } \\
& \qquad \qquad \quad \text{  some complex analytic }\wt{C}\text{ with }\dim\wt{C} \ge t-r \}, \nonumber
\end{align}
where $\wt{s}_V = \wt{\varphi}(\wt{s})$ for each $\wt{s} \in \wt{S}$, and 
\begin{equation}\label{EqBettiMapBettiFoliation2}
u_S^{-1}\left(S_{\mathrm{rk}}(t)\right) = \{\wt{s} \in \wt{S} : \mathrm{rank} (\mathrm{d}\wt{\varphi}_V)_{\wt{s}} \le t\}.
\end{equation}
\end{lemma}
In the union in \eqref{EqBettiMapBettiFoliation1}, the second condition  $\{\wt{s}_V\} \times \wt{C} \subseteq \wt{\varphi}(\wt{S})$ always holds for $r> t$.
 
\begin{proof}
By the characteriation of the leaves of $\CF_{\mathrm{Betti}}$ below \eqref{EqPeriodJacobian}, we have  
\[
u_{\sJ}^{-1}\left( \nu(S^{\mathrm{Betti}}(t))\right) = \{ \wt{s}' \in \wt{\nu}(\wt{S}) : \dim_{\wt{s}'} \wt{\varphi}_{\sJ}^{-1}(\{\wt{s}'_V\} \times \sM_0\}) \ge  t\},
\]
 where $\wt{s}'_V$ is the image of $\wt{s} \in \wt{\sJ} \xrightarrow{\wt{\varphi}_{\sJ}} V_0(\BR) \times \sM_0 \rightarrow V_0(\BR)$ with the last map being the natural projection. 
Applying $\wt{\nu}^{-1}$ (whose fibers are of dimension $0$ because $\nu$ is injective) to the set above and noticing that $u_{\sJ}\circ \wt{\nu}=\nu\circ u_S$, we have
\[
u_S^{-1}\left(S^{\mathrm{Betti}}(t)\right) = \{\wt{s} \in \wt{S} : \dim_{\wt{s}} (\wt{\varphi}_{\sJ}\circ \wt{\nu})^{-1}(\{\wt{s}'_V\} \times \sM_0\}) \ge t\}.
\] 
Hence the first equality in \eqref{EqBettiMapBettiFoliation1} holds because $ \wt{\varphi}_{\sJ}\circ \wt{\nu}=\wt{\varphi} $ and by Lemma~\ref{LemmaStrucMTDomain}. Similarly we can prove \eqref{EqBettiMapBettiFoliation2}.

The second equality in \eqref{EqBettiMapBettiFoliation1} clearly holds. We are done.
\end{proof}

We also make the following observation. Lemma~\ref{LemmaStrucMTDomain} and the discussion around \eqref{EqPeriodJacobian} imply the following: $\sJ_{\nu} := S\times_{\Gamma_0\backslash\sD_0}(\Gamma\backslash\sD) \subseteq \sJ(\BV_{\BZ})$ 
is the relative intermediate Jacobian of a sub-VHS 
 translated by a torsion multisection; it contains $\nu(S)$ and is the minimal one containing $\nu(S)$ with respect to inclusion. The relative dimension $\dim \sJ_{\nu} - \dim S$ equals $\frac{1}{2}\dim V$. Hence we have the following geometric characterization of $V$.
\begin{rmk}\label{RmkSubTorusFib}
$V$  equals $\BV'_{\BQ,s}$ for any $s \in S(\BC)$, where  $\BV'_{\BZ}$ is the largest sub-VHS of $\BV_{\BZ}$ such that $\nu$ is torsion under the projection $\sJ(\BV_{\BZ}) \rightarrow \sJ(\BV_{\BZ}/\BV'_{\BZ})$.
\end{rmk}
 
Since $\wt{\varphi}_V$ factors through $\wt{\varphi}$ and has target $V(\BR)$, we have the following trivial upper bound  by \eqref{EqBettiMapBettiFoliation2}
\begin{equation}\label{EqBettiRankTrivialUB}
\nu_{\mathrm{Betti},s}(T_s S) \le \min\left\{ \dim \varphi(S) , \frac{1}{2}\dim V \right\} \quad\text{ for all }s \in S(\BC).
\end{equation}
This trivial upper bound can furthermore be improved when $\sJ(\BV_{\BZ})\rightarrow S$ is not an abelian scheme by Griffiths' transversality.

\section{The Betti rank and Zariski closedness of the Betti strata}\label{SectionBettiRank}

The goal of this section is to prove Theorem~\ref{ThmZarClosedIntro} and Theorem~\ref{ThmBettiRankIntro}. 
Retain the notation of $\S$\ref{SubsectionResultsGeom}. 
Let $S$ be a smooth irreducible quasi-projective variety. Let $(\BV_{\BZ},\sF^{\bullet})$ be a polarized VHS on $S$ of weight $-1$. Let $\nu \colon S \rightarrow  \sJ(\BV_{\BZ})$ be an admissible normal function. We have the Betti foliation $\CF_{\mathrm{Betti}}$ on $\sJ(\BV_{\BZ})$, which induces the linear map \eqref{DefnBettiIntro} $$
\nu_{\mathrm{Betti},s} \colon T_s S \rightarrow T_{\nu(s)} \sJ(\BV_{\BZ,s})$$
at each $s \in S(\BC)$, and the Betti rank \eqref{EqBettiRankIntro} 
\begin{equation*}
r(\nu) = \max_{s \in S(\BC)} \dim\nu_{\mathrm{Betti},s}(T_s S).
\end{equation*}

 Retain the notation from \eqref{EqDiagramPeriodMapComplete}. In particular, we have the period map  $\varphi = \varphi_{\nu} \colon S \rightarrow \Gamma\backslash\sD$ for the (mixed) Mumford--Tate domain $\sD$, the $\BQ$-group $\mathbf{G}=\MT(\sD)$ and its unipotent radical $V$ which is a vector group. We emphasize that $V$ is, in general, not a fiber of $\BV_{\BQ}$.

We prove the following formula for $r(\nu)$. In practice, for some families it is not hard to determine the generic Mumford--Tate group $\mathbf{G}$ and its normal subgroups, and then this formula makes the computation of $r(\nu)$ practical.
\begin{thm}\label{ThmBettiRankFormulaMain}
The following formula gives the Betti rank:
\begin{equation}\label{EqBettiRankFormula}
r(\nu) = \min_{N} \left\{ \dim \varphi_{/N}(S) + \frac{1}{2}\dim_{\BQ}(V\cap N) \right\},
\end{equation}
where $N$ runs through the set of normal subgroups of $\mathbf{G}$, and $\varphi_{ /N}$ is the induced map
\[
\varphi_{/N} \colon S \xrightarrow{\varphi} \Gamma\backslash\sD \xrightarrow{[p_N]} \Gamma_{/N}\backslash (\sD/N)
\]
with $[p_N]$  the quotient defined in \eqref{EqQuotientMTDomain}.
\end{thm}

Taking $N = \{1\}$ and $N=\mathbf{G}$, we recover the trivial upper bound  \eqref{EqBettiRankTrivialUB} of $r(\nu)$. 


We also show that the Betti foliation defines Zariski's closed strata on $S$.
\begin{thm}\label{ThmZariskiClosedDegLoci}
For each $t \ge 0$, the set 
\[
S^{\mathrm{Betti}}(t) := \{s \in S(\BC): \dim_{\nu(s)} (\nu(S) \cap \CF_{\mathrm{Betti}}) \ge t \}
\]
is Zariski closed in $S$. In particular, $r(\nu) = \dim S - \min\{ t \ge 0 : S^{\mathrm{Betti}}(t) = S\}$.
\end{thm}


The proofs of Theorem~\ref{ThmBettiRankFormulaMain} and Theorem~\ref{ThmZariskiClosedDegLoci} are simultaneous and follow the guideline of the first-named author's \cite{GaoBettiRank} on the case when $\sJ(\BV_{\BZ})$ is polarizable. 
 A key ingredient for our proof is the mixed Ax--Schanuel theorem 
 for admissible VMHS   independently proved by Chiu \cite{KennethChiuAS} and Gao--Klingler \cite{GKAS}. 
We also invoke  \cite{BBTmixed} on the algebraicity of $\varphi(S)$, whose proof builds up on o-minimal GAGA, 
to ease the notation for the proof. 
 Another crucial input for the proof is the o-minimal structure associated with the period map \cite{BBKT}, which we recalled in $\S$\ref{SubsectionOminimality}.  This allows us to apply (o-minimal) definable Chow.


\subsection{Replacing $S$ by $\varphi(S)$}\label{SubsectionReplacingSbyImage}
We shall replace $S$ by $\varphi(S)$ in our proof using \cite{BBTmixed}. This largely eases the notation.

By the main result of \cite{BBTmixed}, the period map $\varphi = \varphi_{\nu} \colon S \rightarrow \Gamma\backslash \sD$ factors as $S \rightarrow S' \xrightarrow{\iota} \Gamma\backslash\sD$, with $S \rightarrow S'$ a dominant morphism between algebraic varieties and $\iota$ an immersion in the category of complex varieties. Then $\iota$ induces an integral admissible VMHS on $S'$ for which $\iota$ is the period map. By abuse of notation, we use $\varphi\colon S \rightarrow S'$ and see $\iota$ as an inclusion. We have the following diagram.
\[
\xymatrix{
& S' \ar@{^(->}[r] \ar@{-->}[rd] & \Gamma\backslash\sD \ar[d]^-{[p]} \\
S \ar[rr]^-{\varphi_0} \ar[ru]^-{\varphi} & & \Gamma_0\backslash\sD_0
}
\]
with the dotted arrow being the restriction $[p]|_{S'}$. 

Recall from Remark~\ref{RmkSubTorusFib} that $\sJ_{\nu} := S\times_{\Gamma_0\backslash\sD_0}(\Gamma\backslash\sD)$ is an intermediate Jacobian over $S$. Set $\sJ' := S' \times_{\Gamma_0\backslash\sD_0}(\Gamma\backslash\sD)$; it is the translate of an intermediate Jacobian  over $S'$ by a torsion section. Then the inclusion $S' \subseteq \Gamma\backslash\sD$ yields a section $\nu'$ of $\sJ' \rightarrow S'$, and thus we can define $S^{\prime \mathrm{Betti}}(t)$ for each $t\ge 0$ with respect to $\nu'$. We have the following commutative diagram, with $\varphi_{\sJ}$ induced by $\varphi$, such that $\varphi_{\sJ} \circ \nu = \nu' \circ \varphi$:
\[
\xymatrix{
\sJ_{\nu} \ar[r]^-{\varphi_{\sJ}} \ar[d] & \sJ' \ar[d] \\
S \ar[r]^-{\varphi} \ar@/^1pc/[u]^-{\nu} & S' \ar@/_1pc/[u]_-{\nu'}.
}
\]

For each $r \ge 0$, denote by 
\begin{equation}\label{EqPeriodMapLargeFiberPart}
S_{\ge r} := \{s \in S(\BC) : \dim_s \varphi^{-1}(\varphi(s)) \ge r\}.
\end{equation}
It is a closed algebraic subset of $S$ by upper semi-continuity.

By (the proof of) Lemma~\ref{LemmaBettiFoliationBettiRank}, more precisely the second equality of \eqref{EqBettiMapBettiFoliation1}, we have
\begin{equation}\label{EqSCFSprimeCF}
S^{\mathrm{Betti}}(t) = S_{\ge t} \cup  \bigcup\nolimits_{0\le r \le t-1} S_{\ge r} \cap \varphi^{-1}\left(S^{\prime \mathrm{Betti}}(t-r)\right).
\end{equation}
This equality allows us to replace $S$ by $\varphi(S)$ to study the Betti rank and the Betti strata.

\subsection{Bi-algebraic system and Ax--Schanuel}

From now on, in the whole section, we replace $S$ by $\varphi(S)$ and view $S$ as an algebraic subvariety of the complex analytic space $\Gamma\backslash\sD$, unless otherwise stated.

Let $\wt{Y} \subseteq \sD$ (resp.  $Y \subseteq S$) be a complex analytic irreducible subset (resp. an irreducible subvariety).
\begin{defn}
\begin{enumerate}
\item[(i)]  The \textbf{weakly special closure} of $\wt{Y}$, denoted by $\wt{Y}^{\mathrm{ws}}$, is the smallest weak Mumford--Tate domain in $\sD$ which contains $\wt{Y}$.
\item[(ii)] The \textbf{weakly special closure} of $Y$, denoted by  $ Y^{\mathrm{ws}}$, is $u(\widetilde{Y}^{\mathrm{ws}})$ for one (hence any) complex analytic irreducible component $\widetilde{Y}$ of $u^{-1}(Y)$.
\end{enumerate}
\end{defn}

The following Ax--Schanuel theorem for VMHS was independently proved by Chiu \cite{KennethChiuAS} and Gao--Klingler \cite{GKAS}.

\begin{thm}[weak Ax--Schanuel for VMHS]\label{ThmAS}
Let $\wt{Z} \subseteq u^{-1}(S)$ be a complex analytic irreducible subset. Then
\begin{equation}\label{EqWAS}
\dim \wt{Z}^{\mathrm{Zar}} + \dim u(\wt{Z})^{\mathrm{Zar}}  \ge \dim  \wt{Z}^{\mathrm{ws}} + \dim \wt{Z},
\end{equation}
where $\wt{Z}^{\mathrm{ws}}$ is the smallest weak Mumford--Tate domain which contains $\wt{Z}$.
\end{thm}

We close this introductory subsection with the following definition. In practice, we often need to work with algebraic subvarieties $Y\subseteq S$, which are not weak Mumford--Tate domains, and the following number measures how far it is from being one.
\begin{equation}
\delta_{\mathrm{ws}}(Y) := \dim  Y^{\mathrm{ws}} - \dim Y.
\end{equation}
If we do not replace $S$ by $\varphi(S)$, then each $Y$ on the right-hand side should be replaced by $\varphi(Y)$.

\begin{defn}\label{DefnWOP}
An irreducible algebraic subvariety $Y$ of $S$ is called {\em weakly optimal} if the following holds: $Y \subsetneq Y' \subseteq S \Rightarrow \delta_{\mathrm{ws}}(Y) < \delta_{\mathrm{ws}}(Y')$, for any $Y' \subseteq S$ irreducible.
\end{defn}

\subsection{Applications of Ax--Schanuel}
Retain the notation in \eqref{EqDiagramPeriodMapComplete}. We start with the following application of mixed Ax--Schanuel.
\begin{prop}\label{PropBettiRankNonMaxAndDegLocus}
For each $t \ge 0$,  
$S^{\mathrm{Betti}}(t)$
is contained in the union of weakly optimal subvarieties $Y \subseteq S$ satisfying
\begin{equation}\label{EqDegenerateMemberUnion}
\dim Y \ge \dim Y^{\mathrm{ws}} -  \dim [p](Y^{\mathrm{ws}}) +t.
\end{equation}
\end{prop}
\begin{proof}
It suffices to prove two things:
\begin{enumerate}
\item[(i)] $S^{\mathrm{Betti}}(t)$ is covered by the union of irreducible subvarieties $Y \subseteq S$ satisfying \eqref{EqDegenerateMemberUnion} (without requiring $Y$ to be weakly optimal);
\item[(ii)] If $Y \subseteq S$ is an irreducible subvariety satisfying \eqref{EqDegenerateMemberUnion} and is maximal for this property with respect to inclusions, then $Y$ is weakly optimal.
\end{enumerate}

Let us prove (i). By Lemma~\ref{LemmaBettiFoliationBettiRank}, more precisely the first equality of \eqref{EqBettiMapBettiFoliation1}, $S^{\mathrm{Betti}}(t) $ is covered by irreducible subvarieties $Y \subseteq S$ such that
\[
Y:= \overline{u(\{a\} \times \wt{C})}^{\mathrm{Zar}},
\]
 for some complex analytic irreducible $\wt{C} \subseteq \sD_0$ with $\dim \wt{C} = t$ and some $a \in V(\BR)$.



Apply mixed Ax--Schanuel in this context (Theorem~\ref{ThmAS} to $\{a\} \times \wt{C}$) and recall that we are using $S = \varphi(S)$ here. Then we get
\[
\dim \overline{\{a\} \times \wt{C}}^{\mathrm{Zar}} + \dim Y\ge \dim ( \{a\} \times \wt{C} )^{\mathrm{ws}} + t.
\]
By Lemma~\ref{CorZariskiClosureBettiFiber}, $\overline{\{a\} \times \wt{C}}^{\mathrm{Zar}} = \{a\} \times \overline{\wt{C}}^{\mathrm{Zar}}$. Hence
\[
\dim Y  \ge \dim ( \{a\} \times \wt{C})^{\mathrm{ws}} -  \dim \overline{\wt{C}}^{\mathrm{Zar}} + t  \ge  \dim  ( \{a\} \times \wt{C} )^{\mathrm{ws}}  -  \dim \wt{C}^{\mathrm{ws}} + t .
\]
The last inequality holds because $\overline{\wt{C}}^{\mathrm{Zar}} \subseteq \wt{C}^{\mathrm{ws}}$. So
\[
\dim Y  \ge \dim  ( \{a\} \times \wt{C} )^{\mathrm{ws}}  -  \dim  \wt{C}^{\mathrm{ws}} +t.
\]
Now, to prove (i), it suffices to prove $\dim   ( \{a\} \times \wt{C} )^{\mathrm{ws}}  = \dim   Y^{\mathrm{ws}}$ and $\dim  \wt{C}^{\mathrm{ws}} = \dim [p](  Y^{\mathrm{ws}})$.

Let us prove $u( ( \{a\} \times \wt{C} )^{\mathrm{ws}} ) =  Y^{\mathrm{ws}}$; the upshot is $\dim  ( \{a\} \times \wt{C} )^{\mathrm{ws}}  = \dim   Y^{\mathrm{ws}}$. By definition of $Y$, we have $u(\{a\} \times \wt{C}) \subseteq Y$. Hence $u( ( \{a\} \times \wt{C} )^{\mathrm{ws}}  ) \subseteq  Y^{\mathrm{ws}}$. On the other hand, $u( ( \{a\} \times \wt{C} )^{\mathrm{ws}}  )$ is closed algebraic by \cite[Cor.~6.7]{BBKT} (which is a consequence of the o-minimal setup explained in $\mathsection$\ref{SubsectionOminimality} and definable Chow). So $Y \subseteq u( ( \{a\} \times \wt{C} )^{\mathrm{ws}}  )$. So $ Y^{\mathrm{ws}} \subseteq u( ( \{a\} \times \wt{C} )^{\mathrm{ws}} )$. Now we have established $u( ( \{a\} \times \wt{C} )^{\mathrm{ws}} ) =  Y^{\mathrm{ws}}$.

Similarly, we have $\dim  \wt{C}^{\mathrm{ws}} = \dim [p](  Y^{\mathrm{ws}})$. Hence, we are done for (i). 

For (ii), let $Y \subseteq Y' \subseteq S$. Assume $\delta_{\mathrm{ws}}(Y) \ge \delta_{\mathrm{ws}}(Y')$, \textit{i.e.} 
\[
\dim  Y^{\mathrm{ws}} - \dim Y \ge \dim  Y^{\prime,\mathrm{ws}} - \dim Y'.
\]
The assumption on $Y$ implies $\dim  Y^{\mathrm{ws}} - \dim Y \le \dim [p](   Y^{\mathrm{ws}} ) -t$. Combined with the inequality above, we obtain $\dim  Y^{\prime,\mathrm{ws}} - \dim Y' \le  \dim [p](   Y^{\prime,\mathrm{ws}} ) -t$ because $Y \subseteq Y'$. Therefore $Y = Y'$ by maximality of $Y$. So $Y$ is weakly optimal by definition. Hence, (ii) is established.
\end{proof}

Next we apply the following finiteness result called {\it Geometric Zilber--Pink}, proved by Baldi--Urbanik \cite[Thm.~7.1]{BaldiUrbanik}, in the case of weight $-1$ and $0$. This proposition itself is an application of mixed Ax--Schanuel. The blueprint was laid by Ullmo \cite{UllmoApp}, and the current version of the finiteness statement was proved by Daw--Ren \cite{DawRenAppOfAS} for pure Shimura varieties and by the first-named author \cite[Thm.~1.4]{GaoMixedAS} for mixed Shimura varieties of Kuga type, \textit{i.e.} when $\sJ(\BV_{\BZ})$ is polarizable. Passing from the Shimura case to variations of (mixed) Hodge structures, it was proved by Baldi--Klingler--Ullmo \cite[$\mathsection$6]{BaldiKlinglerUllmo} for the pure case before the final version for VMHS was proved by Baldi--Urbanik \cite[Thm.~7.1]{BaldiUrbanik}. We point out that weakly optimal subvarieties are called {\it monodromically atypical maximal} in  \cite{BaldiKlinglerUllmo} and \cite{BaldiUrbanik}.
\begin{prop}\label{PropFinitenessBogomolov}
There exist finitely many pairs $(\sD_1,N_1), \ldots, (\sD_k,N_k)$, with each $\sD_j$ a Mumford--Tate domain contained in $\sD$ and $N_j$ a normal subgroup of $\mathrm{MT}(\sD_j)$, such that the following holds. For each weakly optimal subvariety $Y \subseteq S$,  $ Y^{\mathrm{ws}}$ is the image of an $N_j(\BR)^+$-orbit contained in $\sD_j$ under $u \colon \sD \rightarrow \Gamma\backslash\sD$ for some $j \in \{1,\ldots,k\}$.
\end{prop}

Denote by $\Gamma_j = \Gamma \cap \mathrm{MT}(\sD_j)(\BQ)$ and $\Gamma_{j,/N_j} = \Gamma_j / (\Gamma_j \cap N_j(\BQ))$. Then equivalently, each such $ Y^{\mathrm{ws}}$ is a fiber of the quotient $[p_{N_j}] \colon u(\sD_j) = \Gamma_j \backslash\sD_j \rightarrow \Gamma_{j,/N_j}\backslash (\sD_j/N_j)$.  

In $\mathsection$\ref{SubsectionOminimality}, we endowed $\Gamma\backslash\sD$ with a semi-algebraic structure, and hence $\Gamma_j \backslash\sD_j$ with a semi-algebraic structure. We can endow $\Gamma_{j,/N_j}\backslash (\sD_j/N_j)$ with a semi-algebraic structure  in a similar way. Then $[p_{N_j}]$ is semi-algebraic because the quotient map $\sD_j \rightarrow \sD_j/N_j$ is.


\subsection{Proof of Theorem~\ref{ThmBettiRankFormulaMain}}
For each $j \in \{1,\ldots,k\}$, the set  $u(\sD_j)\cap S$ is a closed algebraic subset of $S$ by \cite[Cor.~6.7]{BBKT}; this is a consequence of the o-minimal setup explained in $\mathsection$\ref{SubsectionOminimality} and definable Chow.
 The restriction 
\[
[p_{N_j}]|_S \colon u(\sD_j) \cap S \rightarrow \Gamma_{j,/N_j}\backslash (\sD_j/N_j)
\]
is both complex analytic and  definable; see $\mathsection$\ref{SubsectionOminimality}.

For each $t \ge 0$, the subset
\begin{equation}\label{EqEjSet}
E_j(t) := \left\{ s\in u(\sD_j) \cap S : \dim_s [p_{N_j}]|_S^{-1}([p_{N_j}](s)) \ge \frac{1}{2}\dim (V \cap N_j) +t \right\}
\end{equation}
is both definable and complex analytic in $u(\sD_j)\cap S$. Hence, $E_j(t)$ is algebraic by definable Chow. Moreover, it is closed in $u(\sD_j)\cap S$ by the upper semi-continuity of fiber dimensions. So $E_j(t)$ is a closed algebraic subset of $S$.



\begin{prop}\label{PropFoliationDegLociSubset}
For each $t \ge 0$, we have
\[
S^{\mathrm{Betti}}(t) \subseteq \bigcup\nolimits_{j=1}^k E_j(t).
\]
\end{prop}
\begin{proof}
%
%
%
Let $t \ge 0$. 
By Proposition~\ref{PropBettiRankNonMaxAndDegLocus}, $S^{\mathrm{Betti}}(t)$ 
 is covered by weakly optimal $Y \subseteq S$ such that 
 $\dim Y \ge  \dim  Y^{\mathrm{ws}} -  \dim [p]( Y^{\mathrm{ws}}) +t$.  Then by Proposition~\ref{PropFinitenessBogomolov}, $ Y^{\mathrm{ws}}$ is a fiber $[p_{N_j}]$ for some $j \in \{1,\ldots,k\}$, and hence $\dim  Y^{\mathrm{ws}} - \dim [p]( Y^{\mathrm{ws}}) = \frac{1}{2}\dim (V \cap N_j)$. So $\dim Y \ge  \frac{1}{2}\dim (V \cap N_j)+t$. So $Y \subseteq E_j(t)$ because $[p_{N_j}](Y)$ is a point.
\end{proof}

Now we are ready to prove Theorem~\ref{ThmBettiRankFormulaMain}.
\begin{proof}[Proof of Theorem~\ref{ThmBettiRankFormulaMain}]
In this proof, we go back to our original setting and do not replace $S$ by $\varphi(S)$. We have $S \xrightarrow{\varphi} S' \subseteq \Gamma\backslash\sD$ with $S' = \varphi(S)$ an algebraic subvariety of $\Gamma\backslash\sD$.

Let us prove ``$\le$''. By \eqref{EqBettiMapBettiFoliation2} we have $r(\nu) = \max_{\wt{s} \in\wt{S}} (\mathrm{d}\wt{\varphi}_V)_{\wt{s}}$. Hence $r(\nu) \le \dim \varphi_{/N}(S) + \frac{1}{2}\dim(V\cap N)$ for any normal subgroup $N \lhd \mathbf{G}$.

Let us prove ``$\ge$''.  Let $t = \dim S - r(\nu)$. Then $S^{\mathrm{Betti}}(t)$ contains a non-empty open subset of $S^{\mathrm{an}}$.
By \eqref{EqSCFSprimeCF} and Proposition~\ref{PropFoliationDegLociSubset} (which should be applied to $S^{\prime \mathrm{Betti}}(t-r)$ for each $0\le r \le t-1$), we have 
\[
S^{\mathrm{Betti}}(t) \subseteq S_{\ge t} \cup \bigcup\nolimits_{0\le r\le t-1,~ 1\le j \le k} S_{\ge r} \cap \varphi^{-1}\left(E_j(t-r)\right).
\]
Each $S_{\ge r}$ is Zariski closed in $S$, and each $E_j(t-r)$ is Zariski closed in $S' = \varphi(S)$. Hence, each member in the union on the right-hand side is  Zariski closed in $S$. Taking the Zariski closure of both sides, we then have $S$ equal to a member on the right-hand side.

If $S = S_{\ge t}$, then ``$\ge$'' holds already for $N=\{1\}$. 

Assume $S = S_{\ge r} \cap \varphi^{-1}\left(E_j(t-r)\right)$ for some $0\le r \le t-1$ and some $j$. Then $S= S_{\ge r} =\varphi^{-1}\left(E_j(t-r)\right)$, So each fiber of $\varphi$ has dimension $\ge r$, and $S' = E_j(t-r)$. Moreover, 
$\mathrm{MT}(\sD_j)= \mathbf{G}$ because $S' = \varphi(S)$ is Hodge generic in $\Gamma\backslash\sD$.  
Set $N = N_j$. 
 Each fiber of the map
\[
\varphi_{ /N} \colon S \xrightarrow{\varphi} S' \subseteq \Gamma\backslash\sD \xrightarrow{[p_N]} \Gamma_{/N}\backslash (\sD/N),
\]
has $\BC$-dimension $\ge r + \left( \frac{1}{2}\dim(V\cap N) +(t-r) \right) =  \frac{1}{2}\dim(V\cap N) +t$ by definition of $E_j(t-r)$. So
\[
r(\nu) = \dim S - t \ge \left( \dim \varphi_{/N}(S) + \frac{1}{2}\dim(V\cap N) +t \right) - t = \dim \varphi_{/N}(S) + \frac{1}{2}\dim(V\cap N).
\]
So, ``$\ge$'' is established.
\end{proof}


\subsection{Zariski closedness of the Betti strata}
We start with the following lemma, which is the converse of Proposition~\ref{PropFoliationDegLociSubset}.
\begin{lemma}\label{LemmaEjInDegLocus}
For each $t \ge 0$ and each $j \in \{1,\ldots,k\}$, we have $E_j(t) \subseteq S^{\mathrm{Betti}}(t)$.
\end{lemma}
\begin{proof} Fix $j$. 
Denote by $\mathbf{H}_j = \mathrm{MT}(\sD_j)$, $V_j := V \cap \mathbf{H}_j$, and $\mathbf{H}_{j,0} := \mathbf{H}_j/V_j$.  Under the identification $\sD = V(\BR) \times \sD_0$ below Lemma~\ref{LemmaStrucMTDomain}, we have $\sD_j = (V_j(\BR)+v') \times p(\sD_j)$ for some $v' \in V(\BQ)$ by Lemma~\ref{LemmaStrucMTDomain} (applied to both $\sD_j$ and $\sD$).

Because $N_j \lhd \mathbf{H}_j$, we have: (i) $V \cap N_j = V_j \cap N_j$ is a $\mathbf{H}_{j,0}$-module; (ii) the action of $p(N_j) \lhd \mathbf{H}_{j,0}$ on $V_j / (V_j \cap N_j)$ is trivial. Let $x \in \sD_j$. Under $\sD_j = (V_j(\BR)+v_0) \times p(\sD_j)$, write $x = (v, x_0)$. Then $N_j(\BR)^+x$ becomes $((V\cap N_j)(\BR) + v) \times p(N_j)(\BR)^+x_0$. Notice that this $v \in V(\BR)$ is fixed.

For each $s \in E_j(t)$, by definition there exist an irreducible $\wt{Y} \subseteq u^{-1}(S) \cap \sD_j$ such that $s \in u(\wt{Y})$, $\dim \wt{Y} \ge \frac{1}{2}\dim (V \cap N_j)+t$, and that $\wt{Y}$ is contained in a fiber of the quotient $\sD_j \rightarrow \sD_j/N_j$. The last condition implies that $\wt{Y} \subseteq N_j(\BR)^+x$ for some $x \in \sD_j$. Hence by the discussion above, $\wt{Y} \subseteq ((V\cap N_j)(\BR) +v) \times \sD_0$ for a fixed $v \in V(\BR)$. Now that $\dim_{\BC} \wt{Y} \ge \frac{1}{2}\dim (V \cap N_j) + t $, the following property holds: For each $(a,x_0) \in \wt{Y} \subseteq ((V\cap N_j)(\BR) +v) \times \sD_0$, there exists a complex analytic subset $\wt{C} \subseteq \sD_0$ with $\dim \wt{C} \ge t$ such that $\{a\} \times \wt{C} \subseteq \wt{Y}$. Hence $s \in S^{\mathrm{Betti}}(t)$ by Lemma~\ref{LemmaBettiFoliationBettiRank} (more precisely, the first equality in \eqref{EqBettiMapBettiFoliation1}). Now, the conclusion of the lemma holds as $s$ runs over $E_j(t)$.
\end{proof}

Now we are ready to prove Theorem~\ref{ThmZariskiClosedDegLoci}.
\begin{proof}[Proof of Theorem~\ref{ThmZariskiClosedDegLoci}]
In this proof, we go back to our original setting and do not replace $S$ by $\varphi(S)$. We have $S \xrightarrow{\varphi} S' \subseteq \Gamma\backslash\sD$ with $S' = \varphi(S)$ an algebraic subvariety of $\Gamma\backslash\sD$.

By \eqref{EqSCFSprimeCF}, Proposition~\ref{PropFoliationDegLociSubset} and Lemma~\ref{LemmaEjInDegLocus} (both applied  to $S^{\prime \mathrm{Betti}}(t-r)$ for each $0\le r \le t-1$), we have
\begin{equation}\label{EqDegLocit}
S^{\mathrm{Betti}}(t) = S_{\ge t} \cup \bigcup\nolimits_{0\le r\le t-1,~ 1\le j \le k} S_{\ge r} \cap \varphi^{-1}\left(E_j(t-r)\right).
\end{equation}
So $S^{\mathrm{Betti}}(t)$ is Zariski closed in $S$ because each member in the union on the right-hand side is. The ``In particular'' part is easy to check once we have established the Zariski closedness of $S^{\mathrm{Betti}}(t)$.
\end{proof}

\section{Application to non-degeneracy in some geometric cases}\label{SectionGeomApp}
In this section, we give two applications of Theorem~\ref{ThmBettiRankFormulaMain}: 
when the VHS is irreducible or when it has a simple algebraic monodromy group. Both cases apply to the Gross-Schoen and the Ceresa normal functions. Retain all the notation from $\S$\ref{SectionBettiRank}.

\begin{thm}\label{ThmBignessGeom}
Assume: (i) $(\BV_{\BZ},\sF^{\bullet}) \rightarrow S$ is irreducible, \textit{i.e.} the only sub-VHSs are trivial or itself; (ii) $\nu(S)$ is not a torsion section. Then 
\[
r(\nu) = \min\left\{\dim \varphi(S), \frac{1}{2}\dim \BV_{\BQ,s}\right\}
\]
for one (and hence for all) $s \in S(\BC)$.

In particular, if futhermore $\dim\varphi(S) = \dim S$ and  $\dim \BV_{\BQ,s} \ge 2 \dim S$, then we have  $r(\nu) = \dim S$.
\end{thm}

\begin{proof}

Set $\mathbf{G}_0 =\mathbf{G}/V$. Then $V$ is a $\mathbf{G}_0$-submodule of $\BV_{\BQ,s}$ for one (and hence all) $s \in S(\BC)$, and $\mathbf{G}_0$ is by definition a subgroup of $\GL(\BV_{\BQ,s})$.

By (i), $\BV_{\BQ,s}$ is irreducible as a $\mathbf{G}_0$-module. By (ii), $V\not=\{0\}$. Hence  $V = \BV_{\BQ,s}$. 
 Thus we have $\mathbf G_0\subset \GL(V)$.
Let $N$ be a normal subgroup of $\mathbf{G}$. Then $V \cap N$ is a $\mathbf{G}_0$-submodule of $V$, and hence gives rise to a sub-VHS of  $(\BV_{\BZ},\sF^{\bullet}) \rightarrow S$. Hence by (i), either $V\cap N =\{0\}$ or $V\cap N = V$. 

In the first case $V\cap N = \{0\}$, we have the image $N_0$ of $N$ acting trivially on $V$. Thus $N_0=0$. This shows that $N$ is trivial.  So $r(\nu) = \dim \varphi (S)$ in this case.

In the second case, $V\cap N = V$.  It is clearly true that 
$$\min_{N,~V\cap N=V} \left\{\dim \varphi_{/N}(S) \right\} + \frac{1}{2}\dim V$$ is attained at $N = \mathbf{G}$ and hence equals $\frac{1}{2}\dim V$. Thus in this case $r(\nu) = \frac{1}{2}\dim V = \frac{1}{2} \dim \BV_{\BQ,s}$ by \eqref{EqBettiRankFormula}.

Now, we are done by combining the two cases above.
\end{proof}

\begin{thm}\label{ThmBignessGeom2}
Assume: (i) the connected algebraic monodromy group $H$ of $(\BV_{\BZ}, \sF^{\bullet}) \rightarrow S$ is simple; (ii) $(\BV_{\BZ},\sF^{\bullet}) \rightarrow S$ has no isotrivial sub-VHS, \textit{i.e.} locally constant VHS. Then 
\[
r(\nu) = \min\left\{\dim \varphi(S), \frac{1}{2}\dim V\right\}.
\]
\end{thm}
An important case where assumption (i) is satisfied is when the generic Mumford--Tate group of $(\BV_{\BZ}, \sF^{\bullet}) \rightarrow S$ is quasi-simple; this follows from Deligne \cite[$\S$7.5]{Del}. 
 
 \begin{proof}
Set $\mathbf{G}_0 =\mathbf{G}/V$.  By Deligne \cite[$\S$7.5]{Del}, $H \lhd \mathbf{G}^{\mathrm{der}}_0$. 

Let $N$ be a normal subgroup of $\mathbf{G}$.  The reductive part $N_0 := N/(V\cap N)$ is a normal subgroup of $\mathbf{G}_0$. Now that $N_0 \cap H$ is a normal subgroup of $H$, by (i), we have that either $N_0 \cap H$ is finite or $H < N_0$.

Assume $H< N_0$. Since $\mathbf{G}_0$ is reductive, we can decompose $V = (V\cap N) \bigoplus (V\cap N)^{\perp}$ as $\mathbf{G}_0$-modules. Now $(V\cap N)^{\perp}$ gives rise to a sub-VHS of $(\BV_{\BZ},\sF^{\bullet})$. But $H$ acts trivially on $(V\cap N)^{\perp}$ because $N_0$ acts trivially on $V/(V\cap N)$. 
So $(V\cap N)^{\perp} = \{0\}$ by (ii). Hence $V\cap N = V$, and hence $V < N$. 
 So in this case we have $$\dim \varphi_{/N}(S) + \frac{1}{2}\dim (V\cap N) \ge \frac{1}{2}\dim V.$$
 The equality obtains when $N=\mathbf G$. So in this case $r(\nu) = \frac{1}{2}\dim V$ by \eqref{EqBettiRankFormula}.

Assume that $N_0\cap H$ is finite. To handle this case, we again use the universal period map of the VHS $\wt{\varphi}_0 \colon \wt{S} \rightarrow \sD_0$ associated with the VHS $\BV_{\BZ} \rightarrow S$ from \eqref{EqDiagramPeriodMapComplete}, where $\sD_0$ is a $\mathbf{G}_0(\BR)^+$-orbit and the natural projection $p \colon \sD \rightarrow \sD_0$ is induced by the quotient $p\colon \mathbf{G} \rightarrow \mathbf{G}_0$. Moreover, we have $\wt{\varphi}_0 = p \circ \wt{\varphi}$.

By logarithmic Ax \cite[Thm.~7.2]{GKAS}, we have $\wt{\varphi}_0(\wt{S}) \subseteq H(\BR)^+x_0$ for some $x_0 \in \sD_0$. So $\wt{\varphi}(\wt{S}) \subseteq p^{-1}(\wt{\varphi}_0(\wt{S})) \subseteq p^{-1}(H(\BR)^+x_0)$. Under $\sD = V(\BR)\times \sD_0$ below Lemma~\ref{LemmaStrucMTDomain}, we have $\wt{\varphi}(\wt{S}) \subseteq V(\BR)\times H(\BR)^+x_0$.


Recall the definition $\varphi_{/N} \colon S \xrightarrow{\varphi} \Gamma\backslash\sD \xrightarrow{[p_N]} \Gamma_{/N}\backslash(\sD/N)$, with  $[p_N]$ induced by $p_N \colon \mathbf{G} \rightarrow \mathbf{G}/N$. Notice that $p_N$ can be decomposed as the composite
\[
p_N \colon \mathbf{G} \xrightarrow{p_{V\cap N}} \mathbf{G}/(V\cap N) \xrightarrow{p_{N_0}} \mathbf{G}/N,
\]
and this induces
\[
[p_N] \colon  \Gamma\backslash\sD \xrightarrow{[p_{V\cap N}]} \Gamma_{/V\cap N}\backslash (\sD/(V\cap N)) \xrightarrow{[p_{N_0}]} \Gamma_{/N}\backslash(\sD/N).
\]
We claim that $\dim ([p_{V\cap N}]\circ \varphi)(S) = \dim ([p_{N_0}] \circ [p_{V\cap N}]\circ \varphi)(S) = \dim \varphi_{/N}(S)$. Indeed, $[p_N]$ is obtained from the composite $p_N \colon \sD\xrightarrow{p_{V\cap N}} \sD/(V\cap N) \xrightarrow{p_{N_0}} \sD/N$, and we are left to prove that each fiber of $p_{N_0}$ restricted to $p_{V\cap N}(\wt{\varphi}(\wt{S}))$ has dimension $0$. 
Denote for simplicity by $V' := V/(V\cap N)$. Then $\sD/(V\cap N) = V'(\BR) \times \sD_0$, and $p_{V\cap N}(\wt{\varphi}(\wt{S})) \subseteq V'(\BR) \times H(\BR)^+x_0$ by the conclusion of the previous paragraph. Let us furthermore decompose $\sD_0$. 
As a reductive group, $\mathbf{G}_0$ is the almost direct product of $Z(\mathbf{G}_0)$ and its simple factors. Both $N_0$ and $H$ are normal subgroups of $\mathbf{G}_0$. So $\mathbf{G}_0 = Z(\mathbf{G}_0) \cdot N_0 \cdot H \cdot H' $ as an almost direct product, where $H'$ is a semisimple subgroup of $\mathbf{G}_0$. Applied to $\sD_0 = \mathbf{G}_0(\BR)^+x_0$, we then have a decomposition $\sD_0 \cong \sD_{N_0} \times \sD_H \times \sD_{H'}$ with $\sD_{N_0} = N_0(\BR)^+x_0$ and $\sD_H = H(\BR)^+x_0$ and $\sD_{H'} = H'(\BR)^+x_0$. Now $\sD/(V\cap N) \cong V'(\BR) \times \sD_{N_0} \times \sD_H \times \sD_{H'}$, and $p_{V\cap N}(\wt{\varphi}(\wt{S})) \subseteq V'(\BR) \times \{z\} \times \sD_H \times \{z'\}$ for some $z \in \sD_{N_0}$ and $z' \in \sD_{H'}$, and $p_{N_0}$ is the projection $V'(\BR) \times \sD_{N_0} \times \sD_H \times \sD_{H'} \rightarrow V'(\BR) \times   \sD_H \times \sD_{H'}$. Hence we establish our claim.

Notice that each fiber of $[p_{V\cap N}]$ has dimension $\frac{1}{2}\dim(V\cap N)$. So $\dim \varphi(S) - \dim ([p_{V\cap N}]\circ \varphi)(S) \le \frac{1}{2}\dim(V\cap N)$. We proved in the previous paragraph $\dim ([p_{V\cap N}]\circ \varphi)(S)  = \dim \varphi_{/N}(S)$. So
\[
\dim \varphi_{/ N}(S) + \frac{1}{2}\dim (V\cap N) \ge \dim \varphi(S).
\]
The equality is attained when $N=1$. So in this case $r(\nu) = \dim \varphi(S)$  by \eqref{EqBettiRankFormula}.

Now, we are done by combining the two cases above.
\end{proof}

We can say more. {\em Griffiths' transversality} says that $\nu_{\mathrm{Betti},s}(u) \in \BV_s^{-1,0}$ for all $s \in S(\BC)$ and $\nu_{\mathrm{Betti},s}$ defined in \eqref{DefnBettiIntro}. So $\nu_{\mathrm{Betti},s}(T_s S) <  \frac{1}{2}\dim V$  if $\sJ(\BV'_{\BZ}) \rightarrow S$ is not an abelian scheme, with $\BV'_{\BZ}$  defined in Remark~\ref{RmkSubTorusFib}. 
Hence $\dim V > 2\dim\varphi(S)$ in the situations of Theorem~\ref{ThmBignessGeom} and Theorem~\ref{ThmBignessGeom2} {\it unless} $\sJ(\BV_{\BZ}') \rightarrow S$ is  an abelian scheme.

\smallskip
By Theorem~\ref{ThmZariskiClosedDegLoci}, $S^{\mathrm{Betti}}(1)$ is a proper Zariski closed subset of $S$ if $r(\nu) = \dim S$. 
Observe that $S^{\mathrm{Betti}}(1)$ contains any complex analytic curve $C \subseteq S$ such that $\nu(C)$ is torsion, because all torsion multi-sections are leaves of the Betti foliation. So there are at most countably many $s \in S(\BC)$ outside $S^{\mathrm{Betti}}(1)$ such that $\nu(s)$ is torsion, and we get:
\begin{cor}\label{CorSF1Torsion}
Under the assumptions of either Theorem~\ref{ThmBignessGeom} or Theorem~\ref{ThmBignessGeom2}. Assume that the period map $\varphi$ is generically finite and that $\dim S \le \frac{1}{2}\dim_{\BQ} V$. Then there exists a Zariski open dense subset $U$ of $S$ such that $\nu(s)$ is torsion for at most countably many $s \in U(\BC)$. Indeed, one can take $U = S\setminus S^{\mathrm{Betti}}(1)$.
\end{cor}
\begin{rmk}\label{RmkSF1Torsion}
In general, $\varphi$ is not necessarily generically finite. Then the conclusion becomes: there exists a Zariski open dense subset $U$ of $S$ such that $\{ s \in U(\BC): \nu(s)\text{ is torsion}\}$ is contained in at most countably many fibers of $\varphi$. Indeed, $\varphi(S)$ is an algebraic variety by \cite{BBTmixed}. Then we can conclude by applying Corollary~\ref{CorSF1Torsion} to $\varphi(S)$ and the induced normal function $\nu'$ (see $\S$\ref{SubsectionReplacingSbyImage} for the construction of $\nu'$).
\end{rmk}

In Corollary~\ref{CorSF1Torsion} and Remark~\ref{RmkSF1Torsion}, the condition ``$\dim S \le \frac{1}{2}\dim_{\BQ}V$'' can be removed if $\sJ(\BV'_{\BZ}) \rightarrow S$ is not an abelian scheme.

\section{Degeneracy loci}\label{SectionDegLoci}
In this section, we introduce and discuss  the degeneracy loci. They are the counterpart of the Betti strata and are useful to study the torsion locus of $\nu(S)$. We continue to use the notation in $\mathsection$\ref{SectionBettiRank} and  make the following assumption:

\vskip 0.2em
\begin{center}
{\tt (Hyp): Each fiber of the period map $\varphi = \varphi_{\nu} \colon S \rightarrow \Gamma\backslash\sD$ is finite and $\dim S > 0$}.
\end{center}
\vskip 0.2em

For each $t \ge 0$, define the {\it $t$-th degeneracy locus} of $S$ to be
\begin{equation}\label{EqDegLoci}
S^{\mathrm{deg}}(t) := \bigcup_{ \substack{Y \subseteq S ,~ \dim Y > 0 \\ \dim Y > \dim Y^{\mathrm{ws}} -  \dim [p](Y^{\mathrm{ws}}) -t }} Y,
\end{equation}
where $Y$ runs over all irreducible subvarieties of $S$. This defines the {\it degeneracy strata} on $S$
\begin{equation}\label{EqDegStrata}
S^{\mathrm{deg}}(0) \subseteq S^{\mathrm{deg}}(1) \subseteq \cdots \subseteq S^{\mathrm{deg}}(t_0)\subseteq \cdots
\end{equation}
The lemma below explains how the degeneracy strata are counterparts of  the Betti strata: in defining the Betti stratum $S^{\mathrm{Betti}}(t)$ we require the number $\dim Y - \dim Y^{\mathrm{ws}} + \dim [p](Y^{\mathrm{ws}})$ to be $\ge t$, while in defining the degeneracy stratum $S^{\mathrm{deg}}(t)$ we require this number to be $ > -t$.

\begin{lemma}\label{LemmaDegLociWOP}
In the definition \eqref{EqDegLoci}, we may furthermore assume $Y$ to be weakly optimal in $S$ (see Definition~\ref{DefnWOP}).
\end{lemma}
An upshot of this lemma is $S^{\mathrm{deg}}(0) = S^{\mathrm{Betti}}(1)$ in view of Proposition~\ref{PropBettiRankNonMaxAndDegLocus}. 

\begin{proof}[Proof of Lemma~\ref{LemmaDegLociWOP}]
Assume $Y \subseteq S$ satisfies 
\[
\dim Y > \dim Y^{\mathrm{ws}} -  \dim [p](Y^{\mathrm{ws}}) -t
\]
 and is maximal for this property. It suffices to prove that $Y$ is weakly optimal.

Let $Y \subseteq Y' \subseteq S$ with $Y'$ irreducible. Assume $\delta_{\mathrm{ws}}(Y) \ge \delta_{\mathrm{ws}}(Y')$, \textit{i.e.} 
\[
\dim  Y^{\mathrm{ws}} - \dim Y \ge \dim  Y^{\prime,\mathrm{ws}} - \dim Y'.
\]
The assumption on $Y$ implies $\dim  Y^{\mathrm{ws}} - \dim Y < \dim [p](   Y^{\mathrm{ws}} ) + t$. Combined with the inequality above, we obtain $\dim  Y^{\prime,\mathrm{ws}} - \dim Y' <  \dim [p](   Y^{\prime,\mathrm{ws}} ) + t$ because $Y \subseteq Y'$. Therefore $Y = Y'$ by the maximality of $Y$. So $Y$ is weakly optimal by definition. We are done.
\end{proof}

Now Geometric Zilber--Pink (Proposition~\ref{PropFinitenessBogomolov}) gives finitely many pairs $(\sD_1,N_1),\ldots,(\sD_k,N_k)$ such that for each weakly optimal $Y$ in $S$, its weakly special closure $Y^{\mathrm{ws}}$ is a fiber of 
\[
[p_{N_j}] \colon \Gamma_j\backslash\sD_j \rightarrow \Gamma_{j,/N_j}\backslash(\sD_j/N_j).
\]
Define for each $j \in \{1,\ldots,k\}$ and $t \ge 0$ the set
\begin{equation}\label{EqFjSet}
F_j(t) := \left\{ s\in u(\sD_j) \cap S : \dim_s [p_{N_j}]|_S^{-1}([p_{N_j}](s)) > \max\{0,\frac{1}{2}\dim (V \cap N_j) - t \} \right\}
\end{equation}
which is both definable and complex analytic in $u(\sD_j)\cap S$, and hence is algebraic by definable Chow. Moreover it is Zariski closed in $S$ by the upper semi-continuity of fiber dimensions.

\begin{thm}\label{ThmDegStrataZarClosed}
We have
\begin{equation}\label{EqUnionDeg}
S^{\mathrm{deg}}(t) = \bigcup_{j=1}^k F_j(t).
\end{equation}
In particular, $S^{\mathrm{deg}}(t)$ is Zariski closed in $S$, and hence the degeneracy strata \eqref{EqDegStrata} is stable by Noetherian condition.
\end{thm}
Before moving on to the proof, let us make the following definition.
\begin{defn}
The {\it degenerate degree} of $S$ is the minimal $t_0 \ge 0$ such that $S^{\mathrm{deg}}(t_0) = S$.
\end{defn}
By definition, we have $t_0 \le \dim S^{\mathrm{ws}} - \dim [p](S^{\mathrm{ws}})$ because we can take $Y$ to be $S$ in \eqref{EqDegLoci} when $t = \dim S^{\mathrm{ws}} - \dim [p](S^{\mathrm{ws}})$.

\begin{proof}[Proof of Theorem~\ref{EqDegStrata}]
By Lemma~\ref{LemmaDegLociWOP}, $S^{\mathrm{deg}}(t)$ 
 is covered by weakly optimal $Y \subseteq S$ such that 
 $\dim Y >  \max\{0, \dim  Y^{\mathrm{ws}} -  \dim [p]( Y^{\mathrm{ws}}) - t\}$.  Then by Proposition~\ref{PropFinitenessBogomolov}, $ Y^{\mathrm{ws}}$ is a fiber $[p_{N_j}]$ for some $j \in \{1,\ldots,k\}$, and hence $\dim  Y^{\mathrm{ws}} - \dim [p]( Y^{\mathrm{ws}}) = \frac{1}{2}\dim (V \cap N_j)$. So $\dim Y > \max\{0,  \frac{1}{2}\dim (V \cap N_j)-t\}$. Notice that $[p_{N_j}](Y)$ is a point because $Y \subseteq Y^{\mathrm{ws}}$. So $Y \subseteq F_j(t)$ by definition of $F_j(t)$. This establishes the direction $\subseteq$.
 
Let us prove $\supseteq$, \textit{i.e.} $F_j(t) \subseteq S^{\mathrm{deg}}(t)$ for each $j$. The proof is similar to Lemma~\ref{LemmaEjInDegLocus}. Fix $j$. Denote by $\mathbf{H}_j=\mathrm{MT}(\sD_j)$ and $V_j:=V\cap \mathbf{H}_j$.  For each $s \in F_j(t)$, there exist an irreducible $Y \subseteq S \cap u(\sD_j)$ such that $s\in Y$, $\dim Y> \max\{0, \frac{1}{2}\dim(V\cap N_j) - t\}$, and that $Y$ is contained in a fiber of the quotient $\Gamma_j\backslash \sD_j \rightarrow \Gamma_{j,/N_j}\backslash(\sD_j/N_j)$. The last condition implies that $Y \subseteq u(N_j (\BR)^+x)$ for some $x \in \sD_j$. So $\tilde{Y}^{\mathrm{ws}} \subseteq N_j(\BR)^+x$, and hence $\dim Y^{\mathrm{ws}} - \dim [p](Y^{\mathrm{ws}}) \le \dim N_j(\BR)^+x - \dim p(N_j(\BR)^+x) = \frac{1}{2}\dim(V\cap N_j)$. 
So $\dim Y > \frac{1}{2}\dim(V\cap N_j) - t \ge \dim Y^{\mathrm{ws}} - \dim [p](Y^{\mathrm{ws}}) - t$. Hence $Y \subseteq S^{\mathrm{deg}}(t)$ because $\dim Y$ is also $>0$. This establishes the direction $\supseteq$. We are done.
\end{proof}

\begin{thm}
The degenerate degree $t_0$ of $S$ can be computed by the formula
\begin{equation}\label{EqT0}
t_0 = \min_{N} \left\{ \frac{1}{2}\dim_{\BQ} (V \cap N) + \dim \varphi_{/N}(S) \right\} - \dim S +1
\end{equation}
where $N$ runs over all normal subgroups of $\mathbf{G}$ such that $\dim S - \dim \varphi_{/N}(S)  > 0$.
\end{thm}
\begin{proof}
We have $S = S^{\mathrm{deg}}(t_0)$. So $S^{\mathrm{deg}}(t_0) = F_{j}(t_0)$ for some $j \in \{1,\ldots,k\}$ by \eqref{EqUnionDeg}. Without loss of generality we may assume $j = 1$. Then $\sD_1 = \sD$, and hence $N \lhd \mathbf{G}$. 
We have $S \subseteq u(\sD)$. Since $S = F_1(t_0)$, by \eqref{EqFjSet} each fiber of $[p_{N_1}]|_S$ has dimension $> \max\{0,\frac{1}{2}\dim (V \cap N_1) - t_0 \}$. Therefore
\[
\dim S - \dim \varphi_{/N_1}(S) > \max \left\{0,\frac{1}{2}\dim (V \cap N_1) - t_0 \right\}.
\]
So $t_0 > \frac{1}{2}\dim (V \cap N_1) + \dim \varphi_{/N_1}(S) - \dim S$. This establishes $\ge$.

Conversely, take $N \lhd \mathbf{G}$ such that $\dim S - \dim \varphi_{/N}S > 0$ and that $\frac{1}{2}\dim_{\BQ} (V \cap N) + \dim \varphi_{/N}(S)$ attains its minimum. Set $t' := \frac{1}{2}\dim_{\BQ} (V \cap N) + \dim \varphi_{/N}(S) - \dim S + 1$. Each fiber of $\varphi_{/N}$ is a Zariski closed subset of $S$, and has an irreducible component $Y$  such that $\dim Y \ge \dim S - \dim \varphi_{/N}(S) > 0$. Moreover, the union of such $Y$'s is Zariski dense in $S$.

Take such a $Y$. Then $Y^{\mathrm{ws}}$ is contained in a fiber of $[p_N]$. So $\dim Y^{\mathrm{ws}} - \dim [p](Y^{\mathrm{ws}}) \le \frac{1}{2}\dim(V\cap N)$. Therefore 
\[
\dim Y^{\mathrm{ws}} - \dim [p](Y^{\mathrm{ws}}) - t' \le \frac{1}{2}\dim(V\cap N) - t' = \dim S - \dim \varphi_{/N}(S) -1 \le \dim Y -1.
\]
So $Y \subseteq S^{\mathrm{deg}}(t')$ by definition \eqref{EqDegLoci}. This shows that $S \subseteq S^{\mathrm{deg}}(t')$. So $t_0 \le t'$.

We are done.
\end{proof}

\section{A  discussion towards the arithmetic situation}\label{SectionFuture}
Let $S$ be a quasi-projective complex variety,  $f \colon X \rightarrow S$  a smooth family of projective varieties of relative dimension $d$, and $N$  a relatively ample line bundle on $X$. Let $Z$ be a family of homologically trivial cycles of codimension $n$ in $X/S$ with $n\le \frac {d+1}2$, which is homologically trivial and primitive on each geometric fiber 
$X_s$ in the sense that $c_1(N_s)^{d+2-2n}\cdot Z_s=0$ in $\Ch^{d+2-n}(X_s)$. 

Assume that $f$, $N$, $Z$ are defined over a number field $k$; in general, any family $X/S$ can be embedded into one defined over a number field $k$ by the Lefschetz principle.

The family $Z$ defines a normal function $\nu = \nu_Z \colon S \rightarrow \sJ ^n(X/S)_{\prim}$, $s \colon \mathrm{AJ}(Z_s)$ for the Abel--Jacobi map constructed by Griffith \cite[$\mathsection$11]{GriffithsAJ}, where $\sJ^n(X/S)_\prim$ is the intermediate Jacobian associate with the following VHS:
$$\BH^n:=\ker \left(c_1(N)^{d+2-2n} \colon R^{2n-1}f_*\BZ (n)\lra R^{2d-2n+3}f_*\BZ(d+2-n)\right).$$


We would like to propose the following conjecture on the rationality of the Betti strata.
\begin{conj} \label{conj-rationality}
The Betti stratum $S^{\mathrm{Betti}}(t)$ is defined over $k$ for each $t \ge 0$.
\end{conj} 
In \cite{GZ}, we proved Conjecture~\ref{conj-rationality} for $t=1$ for the Gross--Schoen cycles and Ceresa cycles for any subvariety $S$ of $\sM_g$. The key point of the proof is the {\it volume identity} \cite[Thm.~5.1]{GZ}, which is an arithmetic result.  We expect this volume identity to be true for any such family $f \colon X\rightarrow S$. We point out that in the particular case where $k$ is replaced by $\IQbar$, some partial results towards Conjecture~\ref{conj-rationality} can be deduced from  \cite{BOU}, whose proof is purely geometric (functional transcendental method). In the case where $\sJ^n(X/S)_{\prim}$ is an abelian variety, Conjecture~\ref{conj-rationality} is known with $k$ replaced by $\IQbar$ by \cite[Prop.~4.2.4]{GaoHab}.

When the period map $\varphi = \varphi_{\nu}$ has finite fibers, we also expect the degeneracy locus $S^{\mathrm{deg}}(t)$ to be defined over $k$ for each $t \ge 0$. This is also proved in the Shimura case by  \cite[Prop.~4.2.4]{GaoHab} when $k$ is replaced by $\IQbar$.

\begin{ques}
Is it true that $Z_s$ is non-torsion in $\Ch^n(X_s)$ for all $s \in (S\setminus S^{\mathrm{Betti}}(1))(\BC)$?
\end{ques}
This question has an affirmative answer for {\it transcendental points} in $S\setminus S^{\mathrm{Betti}}(1)$. More precisely, denote by $S^{\mathrm{amp}}:= S\setminus S^{\mathrm{Betti}}(1)$. Then for any $s \in S^{\mathrm{amp}}(\BC)\setminus S^{\mathrm{amp}}(\IQbar)$, the cycle class $Z_s$ is non-torsion in $\Ch^n(X_s)$. To work with $\IQbar$-points, one needs the degeneracy loci defined in $\mathsection$\ref{SectionDegLoci} as indicated by \cite{GH}.

\bibliographystyle{amsalpha}

\end{document}